\documentclass[11pt]{amsart}

\usepackage [applemac] {inputenc} 
\usepackage[english]{babel} 
\usepackage{amsmath,amsthm,amssymb ,euscript} 
\usepackage{amsmath, amscd} 
\usepackage[all,cmtip]{xy}
\usepackage{rotating}        
\usepackage{graphicx}   
\usepackage[square, numbers]{natbib}     
\usepackage{macros}     
\usepackage{amsxtra}
\theoremstyle{plain}

\begin{document}

\setlength{\baselineskip}{1.2\baselineskip}
\title{Cell decompositions of double Bott-Samelson varieties}
\author{Victor Mouquin}
\address{
Department of Mathematics   \\
The University of Hong Kong \\
Pokfulam Road               \\
Hong Kong}
\email{victor.mouquin@gmail.com}
\date{}
\begin{abstract} 
Let $G$ be a  connected complex semisimple Lie group. Webster and Yakimov have constructed partitions of the double flag variety $G/B \times G/B_-$, where $(B, B_-)$ is a pair of opposite Borel subgroups of $G$, generalizing the Deodhar decompositions of $G/B$. We show that these partitions can be better understood by constructing cell decompositions of a product of two Bott-Samelson varieties $Z_{\bfu, \bfv}$, where $\bfu$ and $\bfv$ are sequences of simple reflections. We construct coordinates on each cell of the decompositions and in the case of a positive subexpression, we  relate these coordinates to regular functions on a particular open subset of $Z_{\bfu, \bfv}$. \\
\indent
Our motivation for constructing cell decompositions of $Z_{\bfu, \bfv}$ was to study a certain natural Poisson structure on $Z_{\bfu, \bfv}$. 
\end{abstract}
\maketitle

\section{Introduction and statement of results}

\subsection{Introduction}
Let $G$ be a connected complex semisimple Lie group and $B$ a Borel subgroup of $G$. In \cite{D}, V. Deodhar defined a family of decompositions of the flag variety $G/B$ into pieces each of which is isomorphic to $\IC^k \times (\IC^*)^m$ for some integers $k, m  \geqslant 0$. The Deodhar decomposition has been used to study the Kazdan-Lusztig polynomials associated to $G$ and total positivity in $G/B$ (see \cite{D, MR}). Let $B_-$ be a Borel subgroup of $G$ opposite to $B$ and consider the 
{\it double flag variety}
$$
(G \times G)/(B \times B_-) \cong G/B \times G/B_-.
$$
In \cite{WY}, Webster and Yakimov generalised the decompositions of Deodhar to the double flag variety and showed that each piece of the decompositions is coisotropic with respect to a naturally defined Poisson structure $\pi$ on $G/B \times G/B_-$. \\ 
\indent
The Deodhar decomposition of $G/B$ can be better understood through a natural cell decomposition on Bott-Samelson varieties (see \cite{Du, H}). Motivated by the problem of better understanding the decompositions by Webster-Yakimov, we introduce { \it double Bott-Samelson varieties $Z_{\bfu, \bfv}$}, where $\bfu=(s_1, ..., s_l)$ and $\bfv=(s_{l+1}, ..., s_n)$ are any two sequences of simple reflections in the Weyl group $W$ of $G$. We construct cell decompositions of $Z_{\bfu, \bfv}$ which give rise to the Deodhar-type decompositions of $G/B \times G/B_-$ by Webster-Yakimov when  $\bfu$ and $\bfv$ are reduced. Coordinates are constructed on each piece of the decompositions, and for those cells corresponding to the so-called positive subexpressions (see Definition \ref{defn of positivity}), we relate these coordinates to certain regular functions on an open subset of $Z_{\bfu, \bfv}$. \\
\indent
This article consists of the first part of the author's PhD thesis at the University of Hong Kong, in which a Poisson structure $\pi_{\bfu, \bfv}$ on $Z_{\bfu, \bfv}$ is also studied. The Poisson structure $\pi_{\bfu, \bfv}$ has the property that the natural multiplication map (see (\ref{theta_ bfu bfv}))
$$
\theta_{\bfu, \bfv}: (Z_{\bfu, \bfv}, \pi_{\bfu, \bfv}) \rightarrow (G/B \times G/B_-, \pi)
$$
is Poisson. In a forthcoming paper, which will be partially based on the second part of the author's PhD thesis, the cell decompositions of $Z_{\bfu, \bfv}$ and the coordinates there on  will be used to study the Poisson structure $\pi_{\bfu, \bfv}$.

\subsection{The double Bott-Samelson variety $Z_{\bfu, \bfv}$}

Throughout this article, if $K$ is a Lie group, $Q$ a closed subgroup of $K$, and $K_1, ..., K_n$ submanifolds of $K$ invariant under left and right multiplications by elements in $Q$, then 
$$
K_1 \times_Q \cdots \times_Q K_n /Q
$$
denotes the quotient of $K_1 \times \cdots \times K_n$ by the right action of $Q^n$ defined by 
\begin{equation} \label{times_Q}
(k_1, ..., k_n) \cdot (q_1, ..., q_n) = (k_1q_1, \; q_1^{-1}k_2q_2, \; ..., \;q_{n-1}^{-1}k_nq_n), \;\;\; k_i \in K_i, q_j \in Q.
\end{equation}
For two sequences $\bfu = (s_1, ..., s_l)$, $\bfv = (s_{l+1}, ..., s_n)$ of simple reflections in the Weyl group $W$ of $G$, let $Z_{\bfu}$ and $Z_{-\bfv}$ be the Bott-Samelson varieties 
$$
Z_{\bfu} = P_{s_1} \times_{B} \cdots \times_B P_{s_l}/B, \;\; \;\;\;Z_{-\bfv} = P_{-s_{l+1}} \times_{B_-} \cdots \times_{B_-} P_{-s_n}/B_-,
$$
where $P_s = B \cup BsB$ and $P_{-s} = B_- \cup B_-sB_-$ for a simple reflection $s$. Let
$$
Z_{\bfu, \bfv}=Z_{\bfu} \times Z_{-\bfv}
$$
and call $Z_{\bfu, \bfv}$ the \textit{double Bott-Samelson variety} associated to $(\bfu, \bfv)$. The multiplication in $G \times G$ induces a morphism (see (\ref{theta_ bfu bfv}))
$$
\theta_{\bfu, \bfv}: \;\; Z_{\bfu, \bfv} \rightarrow G/B \times G/B_-.
$$
Of particular interest is the open subset
\begin{equation} \label{calO bfu bfv}
\calO^{\bfu, \bfv}=\left( Bs_1B \times_B \cdots \times_B Bs_lB/B \right) \times \left( B_-s_{l+1}B_- \times_{B_-} \cdots \times_{B_-} B_-s_nB_-/B_- \right)
\end{equation}
of $Z_{\bfu, \bfv}$.  When $\bfu, \bfv$ are reduced, $\theta_{\bfu, \bfv}$ induces an isomorphism
between $\calO^{\bfu, \bfv}$ and the product Schubert cell
$$
(Bs_1\cdots s_lB)/B \times (B_- s_{l+1} \cdots s_n B_-)/B_- \subset G/B \times G/B_-
$$
and thus also a birational isomorphism between $Z_{\bfu, \bfv}$ and the product Schubert variety 
$$
\overline{(Bs_1\cdots s_lB)/B} \times \overline{(B_- s_{l+1} \cdots s_n B_-)/B_-} \subset G/B \times G/B_-.
$$
One of our aims is to demonstrate that double Bott-Samelson varieties not only provide 
resolutions of singularities of Schubert varieties, they also have interesting geometry of their own.

\subsection{Cell decompositions of $Z_{\bfu, \bfv}$ associated to shuffles}

Recall that an $(l, n)$-shuffle is an element $\sigma$ of the symmetric group $S_n$ such that 
$$
\sigma(1) < \sigma(2) < \cdots < \sigma(l), \;\;\;\;\; \sigma(l+1) < \sigma(l+2) < \cdots < \sigma(n).
$$
Let $\sigma$ be an $(l, n)$-shuffle. By  a \textit{$\sigma$-shuffled subexpression} of $(\bfu, \bfv)$ we mean a sequence $\gamma = (\gamma_1, \gamma_2, ..., \gamma_n)$, where for each $1 \leqslant j  \leqslant n$, $\gamma_j =s_{\sigma^{-1}(j)}$ or $e$, the identity element of $W$. Let $\Upsilon_{\sigma(\bfu, \bfv)}$ be the set of all $\sigma$-shuffled subexpressions of $(\bfu, \bfv)$. Our first construction,
presented in Section 3, is a decomposition 
\begin{equation}   \label{Zuv-C-gamma}
Z_{\bfu, \bfv} = \bigsqcup_{\gamma \in \Upsilon_{\sigma(\bfu, \bfv)}} C_\sigma^\gamma 
\end{equation}
associated to each $(l, n)$-shuffle $\sigma$, where for each $\gamma \in \Upsilon_{\sigma(\bfu, \bfv)}$, 
$$
C_{\sigma}^{\gamma} \cong \IC^{n-|J(\gamma)|}
$$
for a subset $J(\gamma)$ of $\{1, 2,\ldots, n\}$ associated to $\gamma$. More precisely, for each $(l, n)$-shuffle $\sigma$, we define an embedding of $Z_{\bfu, \bfv}$ into the quotient variety
\begin{equation}  \label{defn of DF_n}
DF_n = (G \times G) \times_{B \times B_-} \cdots \times_{B \times B_-} (G \times G)/(B \times B_-).
\end{equation}
Let $Z_{\bfu, \bfv}^\sigma$ be the image of $Z_{\bfu, \bfv}$ in $DF_n$. Using certain combinatorial data associated to points in $DF_n$ we first arrive at the set-theoretical decomposition 
\begin{equation}   \label{Zuv-C-gamma-1}
Z_{\bfu, \bfv}^\sigma =\bigsqcup_{\gamma \in \Upsilon_{\sigma(\bfu, \bfv)}} C^\gamma.
\end{equation}
We introduce an open affine covering $\{\calO^\gamma \mid \gamma \in \Upsilon_{\sigma(\bfu, \bfv)}\}$ of $Z_{\bfu, \bfv}^\sigma$ and coordinates $\{z_1, z_2,\ldots, z_n\}$ on each $\calO^\gamma$, and we then show that the combinatorially defined set $C^\gamma$ coincides with the subset of $\calO^\gamma$ defined by $z_j = 0$ for $j \in J(\gamma)$.\\
\indent
The image of $\calO^{\bfu, \bfv} \subset Z_{\bfu, \bfv}$ under the embedding of $Z_{\bfu, \bfv}$ into $DF_n$ is $\calO^{\sigma(\bfu, \bfv)}$, where $\sigma(\bfu, \bfv)=(s_{\sigma^{-1}(1)}, \, s_{\sigma^{-1}(2)}, \, ..., s_{\sigma^{-1}(n)})$. The decomposition in 
(\ref{Zuv-C-gamma-1}) then gives rise to the decomposition
\begin{equation}\label{O-distinguished}
{\mathcal O}^{\sigma(\bfu, \bfv)} = \bigsqcup_{\gamma \in \Upsilon_{\sigma(\bfu, \bfv)}^d} C^\gamma \cap {\mathcal O}^{\sigma(\bfu, \bfv)},
\end{equation}
where $\Upsilon_{\sigma(\bfu, \bfv)}^d$ denotes the set of \textit{distinguished $\sigma$-shuffled subexpressions} of $(\bfu, \bfv)$
(see Definition \ref{distinguished subexpr}). We show that for each $\gamma \in \Upsilon_{\sigma(\bfu, \bfv)}^d$,
$$
C^\gamma \cap {\mathcal O}^{\sigma(\bfu, \bfv)} \cong {\mathbb C}^k \times ({\mathbb C}^*)^m
$$
for some integers $k, m  \geqslant 0$ (see Proposition \ref{O cap C^gamma}). When $\bfu$ and $\bfv$ are reduced, we show in Section \ref{rel with WY} that the decomposition  in (\ref{O-distinguished}) corresponds to that of 
$$
(Bs_1\cdots s_lB)/B \times (B_- s_{l+1} \cdots s_n B_-)/B_- \subset G/B \times G/B_-
$$
by Webster and Yakimov under the isomorphisms
$$
{\mathcal O}^{\sigma(\bfu, \bfv)} \cong \calO^{\bfu, \bfv} \cong (Bs_1\cdots s_lB)/B \times (B_- s_{l+1} \cdots s_n B_-)/B_-. 
$$
As in \cite{WY} for the case when $\bfu$ and $\bfv$ are reduced, for each $\gamma \in \Upsilon_{\sigma(\bfu, \bfv)}^d$, we describe in Section 4 the subset 
$C^\gamma \cap {\mathcal O}^{\sigma(\bfu, \bfv)}$ of ${\mathcal O}^{\sigma(\bfu, \bfv)}$ 
using a set of regular functions $\{\psi_{\gamma, j}\}_{j=1, \ldots, n}$ on ${\mathcal O}^{\sigma(\bfu, \bfv)}$ which are defined in terms of generalized minors \cite{double}. When $\gamma$ is \textit{positive} (see  Definition \ref{defn of positivity}), we show that the coordinate functions $\{z_j \mid j \notin J(\gamma)\}$ on $C^\gamma$ and the regular functions $\psi_{\gamma, j}$ are related by monomial transformations.

\subsection{General notation}    \label{all the nota}

If an upper case letter denotes a Lie group, its Lie algebra will be denoted by the corresponding lower case gothic letter. The identity element in any group will be denoted by $e$. \\
\indent
For integers $k \leqslant n$, we denote by $[k,n]$ the set of integers $j$ such that $k \leqslant j \leqslant n$. \\
\indent
If a set $X$ has a right action by a group $L$ and $p: X \rightarrow Y = X/L$ is the projection, we set $p(x) = [x]_Y \in Y$ for $x \in X$.

\section{Lie theory background}

Here we set up the notation from Lie theory, recall the definition of Bott-Samelson varieties and basic facts about the double flag variety of a complex semisimple Lie group.

\subsection{Notation}

Recall that $G$ is a connected complex semisimple Lie group, and that we have fixed a pair $(B, B_-)$ of opposite Borel subgroups. Let $H = B \cap B_-$ be the maximal torus defined by $B$ and $B_-$. Let $N$ and $N_-$ be respectively the unipotent radicals of $B$ and $B_-$. Let $\triangle$ be the root system defined by $H$, and let $\triangle_+$ and $\Gamma$ be respectively the sets of positive and simple roots defined by $B$. Let $\g = \gh \oplus \sum_{\alpha \in \triangle} \g_{\alpha}$ be the root space decomposition of $\g$. Let $\langle , \rangle_{\g}$ be a fixed nonzero multiple of the Killing form on $\g$. Recall that the restriction of $\langle , \rangle_{\g}$ to $\gh$ is nondegenerate and  defines a nondegenerate symmetric bilinear form on $\gh^*$, which will be denoted by $\langle , \rangle$. For $\alpha \in \triangle_+$, let $h_{\alpha} \in \gh$ be the unique element in $[\g_{\alpha}, \g_{-\alpha}]$ such that $\alpha(h_{\alpha}) = 2$. Fix root vectors $e_{\alpha} \in \g_{\alpha}$ and $e_{-\alpha} \in \g_{-\alpha}$ such that $[e_{\alpha}, e_{-\alpha}] = h_{\alpha}$. Let $\phi_{\alpha}: \gs \gl(2, \IC) \rightarrow \g$ be the Lie algebra homomorphism defined by 
 $$
 \left(\begin{array}{cc}0 & 1 \\0 & 0\end{array}\right) \mapsto e_{\alpha}, \text{ and }
 \left(\begin{array}{cc}0 & 0 \\1 & 0\end{array}\right) \mapsto e_{-\alpha}. 
 $$
 The corresponding Lie group homomorphism from $SL(2, \IC)$ to $G$ will also be denoted by $\phi_{\alpha}$. For $z \in \IC$, let 
 $$
 x_{\alpha}(z) = \phi_{\alpha}\left(\begin{array}{cc}1 & z \\0 & 1\end{array}\right), \;\;\;
 x_{-\alpha}(z) = \phi_{\alpha}\left(\begin{array}{cc}1 & 0 \\z & 1\end{array}\right), \text{ and }
 $$
 $$
 \bar{s}_{\alpha} = \phi_{\alpha}\left(\begin{array}{cc}0 & -1 \\1 & 0\end{array}\right). 
 $$
 
 Let $N_G(H)$ be the normaliser subgroup of $H$ in $G$, and let $W = N_G(H) / H$ be the Weyl group of $G$. Let $S = \{s_{\alpha} = \bar{s}_{\alpha}H \mid \alpha \in \Gamma \}$. It is well known that the pair $(W, S)$ forms a Coxeter system. In particular, $W$ is generated by the simple reflections $s_{\alpha}$, $\alpha \in \Gamma$. We denote the action of $W$ on $H$ by conjugation as a right action by $h^w = \dot{w}^{-1}h\dot{w}$, where $h \in H$, $w \in W$ and $\dot{w}$ is any representative of $w$.  \\
 \indent
Let $X^*(H) = \Hom(H, \IC^*)$ and $X_*(H)= \Hom(\IC^*, H)$ be respectively the lattices of characters and co-characters of $H$. We write the action of $\lambda \in X^*(H)$ on $H$ by $h^{\lambda} \in \IC^*$, $h \in H$. There is a natural embedding of $X^*(H)$ in $\gh^*$ and $X_*(H)$ in $\gh$, and the natural pairing $(,):\gh^*\times \gh \rightarrow \IC$ restricts to a perfect pairing of free abelian groups between $X^*(H)$ and $X_*(H)$ which we also denote by $(, )$. Recall that for  $\alpha \in \triangle_+$ the coroot  $\calpha \in X_*(H)$ is given by
$$
\calpha(z) = \phi_{\alpha}\left(\begin{array}{cc}z & 0 \\0 & z^{-1}\end{array}\right), \; z \in \IC^*. 
$$

An element $x \in N_-HN$ can be uniquely written $x = [x]_-[x]_0[x]_+$, with $[x]_- \in N_-$, $[x]_0 \in H$, and $[x]_+ \in N$. \\

\indent
When $G$ is simply connected, denote by $\lambda_{\alpha} \in X^*(H)$ the fundamental weight associated to $\alpha \in \Gamma$. Recall that $\lambda_{\alpha}$ is defined by $(\lambda_{\alpha}, \check{\beta}) = \delta_{\alpha, \beta}$, $\beta \in \Gamma$.   \\

 \indent
The Weyl group $W$ of $G$ is naturally endowed with a partial ordering $\leqslant$,  called the \textit{Chevalley-Bruhat order}, defined by 
$$
v \leqslant w, \text{ if } BvB \subset \overline{BwB}, \;\; v, w \in W,
$$
where $\overline{BuB}$ is the Zariski closure of $BuB$ in $G$, $u \in W$. For $w \in W$, let $l(w)$ be the smallest integer such that $w$ can be written as a product of simple reflections. The function $l: W \rightarrow \IN$ is called the \textit{length function} of $W$. If $w = s_1 \cdots s_{l(w)}$, with $s_j \in S$,  one calls the finite sequence $(s_1, ..., s_{l(w)})$ a \textit{reduced expression} for $w$. Let $R(w)$ be the set of all reduced expressions for $w$.  If $(s_1, ..., s_k) \in R(w)$, define $\bar{w} = \bar{s}_1 \cdots \bar{s}_k \in N_G(H)$. It is well known (see \cite[Proposition 9.3.2]{S}) that $\bar{w}$ does not depend on the choice of the reduced expression.  Let $v, w \in W$. 
If $(s_1, ..., s_k)$ is a reduced expression for $w$, then $v \leqslant w$ if and only if $v = s_{i_1} \cdots s_{i_l}$, where $(i_1< \cdots < i_l)$ is a subsequence of $(1, ..., k)$. \\
\indent
Recall that a \textit{monoid} is a pair $(M, \ast)$, where $M$ is a set, and $\ast: M \times M \rightarrow M$  a binary operation which is associative, and admits an identity element $e \in M$.  The Weyl group has a natural monoid product $\ast$, given by 
$$
s_{\alpha} \ast w = \max \{s_{\alpha}w, w\} = \left\{ \begin{array}{ll}
s_{\alpha}w, & \text{ if } w < s_{\alpha}w   \\
w,                   & \text{ if } s_{\alpha}w < w,
\end{array} \right.  \;\; \alpha \in \Gamma, w \in W,
$$
and 
$$
u \ast w = s_1 \ast (s_2 \ast ( \cdots s_l \ast w)), \;\; u, w \in W,
$$
where $(s_1, ..., s_l)$ is any reduced expression for $u$. \\
\indent
In \cite{HL} two operations on $W$ are introduced. For $w \in W$ and a simple root $s_{\alpha}$, define 
$$
s_{\alpha} \rhd w = \min \{s_{\alpha}w, w\} = \left\{ \begin{array}{ll}
s_{\alpha}w, & \text{ if } s_{\alpha}w < w   \\
w,                   & \text{ if } w < s_{\alpha}w,
\end{array} \right.
$$
and 
$$
w \lhd s_{\alpha} = \min \{ws_{\alpha}, w\} = \left\{ \begin{array}{ll}
ws_{\alpha}, & \text{ if } ws_{\alpha} < w   \\
w,                   & \text{ if } w < ws_{\alpha}.
\end{array} \right.
$$
For $u \in W$ define 
\begin{align*}
u \rhd w  & = s_1 \rhd (s_2 \rhd (\cdots  s_n \rhd w)), \\
w \lhd u & = ((w \lhd s_1) \lhd s_2) \cdots \lhd s_n,
\end{align*}
where $(s_1, ..., s_n)$ is a reduced expression for $u$. These two definitions are independent of the choice of the reduced expression for $u$. The operation $\lhd$ (resp. $\rhd$) is a right (resp. left) monoidal action of $(W, \ast)$ on $W$. See \cite{HL} for more details.

\subsection{Bott-Samelson varieties}

Bott-Samelson varieties have their origins in the papers \cite{BS} of Bott-Samelson and \cite{De} of  Demazure. For more details, see \cite{B, BK}. For any integer $n \geqslant 1$, let
\begin{equation} \label{defn F_pm n}
F_n = G \times_B \cdots \times _B G/B, \;\; F_{-n} = G \times_{B_-} \cdots \times_{B_-} G/B_-.
\end{equation}
Recall that if $s$ is a simple reflection, then $P_s = B \cup BsB$, and $P_{-s} = B_- \cup B_-s B_-$. Let $\bfu = (s_1, ..., s_l)$ be a sequence of simple reflections. Then one has 
\begin{align*}
Z_{\bfu}  & = P_{s_1} \times_B \cdots \times_B P_{s_l}/B \subset F_l,   \\
Z_{-\bfu} & = P_{-s_1} \times_{B_-} \cdots \times_{B_-} P_{-s_l}/B_- \subset F_{-l}. 
\end{align*}
Both $Z_{\bfu}$ and $Z_{-\bfu}$ are smooth projective varieties of complex dimension $l$. Let $\theta_{\bfu}: Z_{\bfu} \rightarrow G/B$ and $\theta_{-\bfu}: Z_{-\bfu} \rightarrow G/B_-$ be the multiplication maps
\begin{align*}
\theta_{\bfu}([p_1, ..., p_l]_{Z_{\bfu}}) & = p_1 \cdots p_l.B,    \\
\theta_{-\bfu}([p_{-1}, ..., p_{-l}]_{Z_{-\bfu}}) & = p_{-1} \cdots p_{-l}.B_-, 
\end{align*}
where $p_i \in P_{s_i}$ and $p_{-i} \in P_{-s_i}$. The image of  $\theta_{\bfu}$ is the Schubert variety $\overline{BuB/B}$, where $u = s_1 \ast \cdots \ast s_l$. If $(s_1, ..., s_l) \in R(u)$, then it is well known that 
$$
\theta_{\bfu}: Z_{\bfu} \rightarrow \overline{BuB/B}
$$
is a proper, surjective, birational isomorphism.

\subsection{The double flag variety and double Bott-Samelson varieties}

The double flag variety 
$$
DF_1 = (G \times G) /(B \times B_-) \cong G/B \times G/B_-
$$
has the natural left $G \times G$ action given by 
$$
(g_1, g_2)\cdot (h_1.B, \; h_2.B_-) = (g_1h_1.B, \; g_2h_2.B_-), \;\; g_i, h_i \in G.
$$ 
Denote by $G_{\diag}$ the diagonal subgroup of $G \times G$, and for any $w \in W$, let $G_{\diag}(w)$ be the $G_{\diag}$-orbit passing through the point $(w B, B_-)$.  It is well known that the map $w \mapsto G_{\diag}(w)$ gives a parametrization of the $G_{\diag}$-orbits in $DF_1$ by $W$. Identifying $G_{\diag}$ with $G$, the stabilizer subgroup of $(w B,   B_-)$ is $B_- \cap wBw^{-1}$. Thus
$$
\dim_{\IC}(G_{\diag}(w)) = | \triangle | - l(w). 
$$
For $h_1, h_2 \in G$,  $(h_1.B, h_2.B_-)$ lies in $G_{\diag}(w)$ if and only if $h_2^{-1}h_1 \in B_-wB$. \\
\indent
For $u, v \in W$, let $\calO^{u,v}$ be the orbit of $B \times B_-$ passing through the point $(u B, v B_-)$. The map $(u,v) \mapsto \calO^{u,v}$ is a parametrization of the $B \times B_-$- orbits in $DF_1$ by $W \times W$. One has 
$$
\dim_{\IC}(\calO^{u,v}) = l(u) + l(v).
$$
For $u,v,w \in W$, define $\calO^{u,v}_w = \calO^{u,v} \cap G_{\diag}(w)$. 

\begin{pro} \label{w <- v^-1 ast u}  \cite{WY}
Let $u,v,w \in W$. Then $\calO^{u,v}_w$ is nonempty if and only if $w \leqslant v^{-1} \ast u$. 
\end{pro}

If $\bfu = (s_1, ..., s_l)$ and $\bfv = (s_{l+1}, ..., s_n)$ are two sequences of simple reflections, recall that we have defined the double Bott Samelson variety associated to $(\bfu, \bfv)$ as  
$$
Z_{\bfu, \bfv} = Z_{\bfu} \times Z_{-\bfv}.
$$
Let $\theta_{\bfu, \bfv}: Z_{\bfu, \bfv} \rightarrow DF_1$ be the map defined by 
\begin{align}  \label{theta_ bfu bfv}  
\theta_{\bfu, \bfv}([p_1, ..., p_l]_{Z_{\bfu}}, [p_{-(l+1)}, ..., p_{-n}&]_{Z_{-\bfv}})  = (p_1 \cdots p_l.B, p_{-(l+1)} \cdots p_{-n}.B_-)
\end{align}
where $p_i \in P_{s_i}$, and $p_{-j} \in P_{-s_j}$. The image of $\theta_{\bfu, \bfv}$ is the product Schubert cell 
$$
\overline{Bs_1 \ast \cdots \ast s_lB/B} \times \overline{B_- s_{l+1} \ast \cdots \ast s_nB_-/B_-}.
$$
If $\bfu$ and $\bfv$ are reduced, then $\theta_{\bfu, \bfv}$ restricts to an isomorphism 
$$
\theta_{\bfu, \bfv} \mid_{\calO^{\bfu, \bfv}} : \; \calO^{\bfu, \bfv} \rightarrow \calO^{u,v},
$$
where $u = s_1 \cdots s_l$, $v = s_{l+1} \cdots s_n$, and $\calO^{\bfu, \bfv}$ has been defined in (\ref{calO bfu bfv}).

\section{Affine charts and cell decompositions of $Z_{\bfu, \bfv}$ associated to shuffles}

If $\bfu = (s_1, ..., s_l)$, $\bfv = (s_{l+1}, ..., s_n)$ are two sequences of simple reflections, we construct for any $(l,n)$-shuffle an embedding of $Z_{\bfu, \bfv}$ into $DF_n$. Using some combinatorial data associated to points in $DF_n$, we give a set-theoretic decomposition 
$$
Z^{\sigma}_{\bfu, \bfv} = \bigsqcup_{\gamma \in \Upsilon_{\sigma(\bfu, \bfv)}} C^{\gamma},
$$
of the image $Z^{\sigma}_{\bfu, \bfv}$ of $Z_{\bfu, \bfv}$ in $DF_n$. We then cover  $Z^{\sigma}_{\bfu, \bfv}$ with $2^n$ open subsets each of which is isomorphic to $\IC^n$, and use this covering to prove that each $C^{\gamma}$ is isomorphic to $\IC^k$, for some $k \in [1,n]$. \\
\indent
In Section 3.8, we discuss the relation between our decompositions of $Z_{\bfu, \bfv}$ and the Deodhar-type decompositions of $G/B \times G/B_-$ constructed in \cite{WY}.

\subsection{Double flag varieties $DF_n$}

Let $n\geqslant 1$ and recall  
$$ 
DF_n = (G \times G) \times_{B \times B_-} \cdots \times_{B \times B_-} (G \times G)/(B \times B_-)
$$
from (\ref{defn of DF_n}). 
Then $DF_n$ is naturally isomorphic to $F_n \times F_{-n}$ via the isomorphism 
\begin{equation} \label{varphi_n}
\varphi_n: DF_n \rightarrow F_n \times F_{-n}, \;  [(g_1,h_1), ..., (g_n,h_n)]_{DF_n} \mapsto \left( [g_1, ..., g_n]_{F_n}, [h_1, ..., h_n]_{F_{-n}} \right). 
\end{equation}
Moreover, $DF_n$ is naturally endowed with a map to $DF_1$, 
\begin{equation} \label{teta_n}
\theta_n: DF_n \rightarrow DF_1, \; [(g_1,h_1), ..., (g_n,h_n)]_{DF_n} \mapsto (g_1 \cdots g_n.B, h_1 \cdots h_n.B_-),
\end{equation}
and for $j \in [1,n-1]$, one has the projection 
\begin{equation} \label{pee n}
\rho_{j,n}: DF_n \rightarrow DF_j, \;  [(g_1,h_1), ..., (g_n,h_n)]_{DF_n} \mapsto  [(g_1,h_1), ..., (g_j,h_j)]_{DF_j}. 
\end{equation}
Let $x \in DF_n$. Let $w_n(x)$ be the unique Weyl group element such that $\theta_n(x) \in G_{\diag}(w_n)$, and for $j \in [1,n-1]$, denote by $w_j(x) \in W$ the unique element such that $\theta_j(\rho_{j,n}(x)) \in G_{\diag}(w_j)$. One thus obtains a map 
\begin{equation} \label{Phi_n}
\Phi_n: DF_n \rightarrow W^n
\end{equation}
by sending $x$ to $\Phi_n(x) = (w_1(x), ..., w_n(x))$.

\subsection{Embeddings of $Z_{\bfu, \bfv}$ into $DF_n$ using shuffles}

\begin{defn} \label{sigma bfu bfv}
Let $n \geqslant 1$ and $n \geqslant l \geqslant 0$ be integers and denote by $S_{l,n}$ the set of all $(l,n)$-shuffles. To any $\sigma \in S_{l,n}$ we associate a sequence 
$$
\epsilon(\sigma) = (\epsilon(\sigma)_1, ..., \epsilon(\sigma)_n)
$$ 
of elements in $\{1, -1\}$ by setting 
$$
\epsilon(\sigma)_{\sigma(j)} = \left\{ \begin{array}{ll}
 -1, & \text{ if } j \in [1,l]  \\
 1, & \text{ if } j \in [l+1,n].
\end{array} \right. 
$$
For notational simplicity, once $\sigma$ is fixed, we will write $\epsilon = (\epsilon_1, ..., \epsilon_n)$ instead of $\epsilon(\sigma) = (\epsilon(\sigma)_1, ..., \epsilon(\sigma)_n)$. Note that there is a one to one correspondance between $S_{l,n}$ and the set of sequences $\epsilon \in \{1, -1\}^n$, with $-1$ appearing exactly $l$ times.  \\
\indent
Let $\bfu = (s_1, ..., s_l)$ and $\bfv = (s_{l+1}, ..., s_n)$ be two sequences of simple reflections and let $\sigma \in S_{l,n}$. Define the sequence 
\begin{equation} \label{delta}
\sigma(\bfu, \bfv) = (\delta_1, ..., \delta_n)
\end{equation}
by letting $\delta_{\sigma(j)} = s_j$, $j \in [1,n]$. 
\end{defn}

Let $\sigma$ be an $(l,n)$-shuffle. Define maps $\mu_{\sigma}^-: G^n \rightarrow G^l$ and $\mu_{\sigma}^+: G^n \rightarrow G^{n-l}$ by
\begin{align} \label{mu_sigma^pm}
\mu_{\sigma}^-(g_1, ..., g_n) & = (g_1\cdots g_{\sigma(1)}, g_{\sigma(1)+1} \cdots g_{\sigma(2)}, ..., g_{\sigma(l-1)+1} \cdots g_{\sigma(l)})    \notag    \\
\mu_{\sigma}^+(g_1, ..., g_n) & = (g_1 \cdots g_{\sigma(l+1)}, g_{\sigma(l+1)+1} \cdots g_{\sigma(l+2)}, ..., g_{\sigma(n-1)+1} \cdots g_{\sigma(n)}).
\end{align}
Both $\mu_{\sigma}^-$ and $\mu_{\sigma}^+$ descend to maps $[\mu_{\sigma}^-]: F_n \rightarrow F_l$ and $[\mu_{\sigma}^+]: F_{-n} \rightarrow F_{-(n-l)}$ given by 
\begin{align} \label{[mu_sigma^pm]}
[\mu_{\sigma}^-] ([g_1, ..., g_n]_{F_n}) & = [g_1\cdots g_{\sigma(1)}, g_{\sigma(1)+1} \cdots g_{\sigma(2)}, ..., g_{\sigma(l-1)+1} \cdots g_{\sigma(l)}]_{F_l}     \notag    \\
[\mu_{\sigma}^+] ([g_1, ..., g_n]_{F_{-n}}) & =  [g_1 \cdots g_{\sigma(l+1)}, g_{\sigma(l+1)+1} \cdots g_{\sigma(l+2)}, ..., g_{\sigma(n-1)+1} \cdots g_{\sigma(n)}]_{F_{-(n-l)}}.
\end{align}
For $j \in [1,n]$, define 
\begin{equation} \label{G^sigma_j}
G^{\sigma}_j = \left\{ \begin{array}{ll}
G \times B_-, & \text{ if } \epsilon_j = -1    \\
B \times G, & \text{ if } \epsilon_j = 1,
\end{array} \right. 
\end{equation}
and 
\begin{equation} \label{F^sigma}
F^{\sigma}_{l, n} = G^{\sigma}_1 \times_{B \times B_-} \cdots \times_{B \times B_-} G^{\sigma}_n/(B \times B_-) \subset DF_n
\end{equation}
For each $j \in [1,n]$, let $G^{\sigma}_{j,L}$ and $G^{\sigma}_{j, R}$ be respectively the first and second component of $G^{\sigma}_j$. Then the diffeomorphism $\varphi_n$, recall (\ref{varphi_n}), restricts to a diffeomorphism 
$$
\varphi_n \mid_{F^{\sigma}_{l, n}}: F^{\sigma}_{l, n} \rightarrow F^{\sigma}_L \times F^{\sigma}_R,
$$
where 
$$
F^{\sigma}_L = G^{\sigma}_{1, L} \times_B \cdots \times_B G^{\sigma}_{n, L}/B \subset F_n, \;\;  F^{\sigma}_R = G^{\sigma}_{1, R} \times_{B_-} \cdots \times_{B_-} G^{\sigma}_{n, R}/B_- \subset F_{-n}.
$$
By composing with $[\mu_{\sigma} ^-] \times [\mu_{\sigma}^+]$, one obtains the diffeomorphism  
\begin{equation} \label{Psi^sigma}
\varphi^{\sigma} =  ([\mu_{\sigma} ^-] \times [\mu_{\sigma}^+]) \circ \left( \varphi_n \mid_{F^{\sigma}_{l, n}} \right): F^{\sigma}_{l, n} \rightarrow F_l \times F_{-(n-l)}.
\end{equation}
Explicitly, $\varphi^{\sigma}$ is given by 
\begin{align*}
\varphi^{\sigma}([(p_1, q_1), ..., & (p_n,q_n)]_{DF_n})  = \left( [(p_1\cdots p_{\sigma(1)}, p_{\sigma(1)+1} \cdots p_{\sigma(2)}, ..., p_{\sigma(l-1)+1} \cdots p_{\sigma(l)}]_{F_l}  \right.    \\
  & \left. [q_1 \cdots q_{\sigma(l+1)}, q_{\sigma(l+1)+1} \cdots q_{\sigma(l+2)}, ..., q_{\sigma(n-1)+1} \cdots q_{\sigma(n)}]_{F_{-(n-l)}} \right) , 
\end{align*}
where $(p_j,q_j) \in G^{\sigma}_j$. \\

Now, let $\bfu = (s_1, ..., s_l)$ and $\bfv = (s_{l+1}, ..., s_n)$ be two sequences of simple reflections and let $\sigma(\bfu, \bfv) = (\delta_1, ..., \delta_n)$. For $j \in [1,n]$ define 
\begin{equation} \label{P^sigma_j}
P^{\sigma}_j = \left\{
\begin{array}{cc}
P_{\delta_j} \times B_-, & \text{ if } \epsilon_j = -1   \\
B \times P_{-\delta_j}, & \text{ if } \epsilon_j = 1 ,
\end{array} \right.
\end{equation}
and let 
\begin{equation} \label{Z^sigma_bfu, bfv}
Z^{\sigma}_{\bfu, \bfv} = P^{\sigma}_1 \times_{B \times B_-} \cdots \times_{B \times B_-} P^{\sigma}_n/(B \times B_-) \subset F^{\sigma}_{l,n}.
\end{equation}
Then $\varphi^{\sigma}$ restricts to a diffeomorphism $\varphi^{\sigma} \mid_{Z^{\sigma}_{\bfu, \bfv}}: Z^{\sigma}_{\bfu, \bfv} \rightarrow Z_{\bfu, \bfv}$. One thus has an embedding 
\begin{equation}   \label{I^sigma_bfu bfv}
I^{\sigma}_{\bfu, \bfv} = (\varphi^{\sigma} \mid_{Z^{\sigma}_{\bfu, \bfv}})^{-1}: Z_{\bfu, \bfv}  \hookrightarrow DF_n.
\end{equation}

\begin{exa} \label{example1}
Suppose that $\bfu = (s_1, s_2, s_3)$, $\bfv = (s_4, s_5)$ and that $\sigma$, written as a product of transpositions, is $(24)(35)$, so that 
$$
\begin{array}{lllllll}
\sigma(\bfu, \bfv) & = & (s_1, & s_4, & s_5, & s_2, & s_3)   \\
\epsilon &                = & (-1,    & 1,      & 1,     & -1,    &-1).
\end{array}
$$
Then 
$$
Z^{\sigma}_{\bfu, \bfv} = (P_{s_1} \times B_-) \times_{B \times B_-}  (B \times P_{-s_4}) \times_{B \times B_-} (B \times P_{-s_5}) \times_{B \times B_-} (P_{s_2} \times B_-) \times_{B \times B_-} (P_{s_3} \times B_-) / B \times B_-.
$$
For $j \in [1,5]$, let $b_j \in B$, $b_{-j} \in B_-$, $p_j \in P_{s_j}$ and $p_{-j} \in P_{-s_j}$.  Then 
$$
\varphi^{\sigma} \left( [(p_1, b_{-1}), (b_2, p_{-4}), (b_3, p_{-5}), (p_2, b_{-4}), (p_3, b_{-5})]_{Z^{\sigma}_{\bfu, \bfv}} \right) = \left( [p_1, b_2b_3p_2, p_3]_{Z_{\bfu}}, [b_{-1}p_{-4}, p_{-5}]_{Z_{-\bfv}} \right).
$$
The embedding $I^{\sigma}_{\bfu, \bfv}$ is given by 
$$
I^{\sigma}_{\bfu, \bfv} \left( [p_1, p_2, p_3]_{Z_{\bfu}}, [p_{-4}, p_{-5}]_{Z_{-\bfv}} \right) = [(p_1, e), (e, p_{-4}), (e, p_{-5}), (p_2, e), (p_3, e)]_{Z^{\sigma}_{\bfu, \bfv}}.
$$
\end{exa}

\subsection{Shuffled subexpressions and the subvarieties $C^{\gamma}$}

Let $\bfu, \bfv$ and $\sigma$ be as in Definition \ref{sigma bfu bfv}. Recall the map $\Phi_n: DF_n \rightarrow W^n$ from (\ref{Phi_n}), and the sequence $\sigma(\bfu, \bfv)$ from (\ref{delta}).

\begin{lem} \label{w_1, ..., w_n}
Let $\bfw = (w_1, ..., w_n) \in W^n$. If $\bfw \in \Phi_n(Z^{\sigma}_{\bfu, \bfv})$, then $w_1 \in \{ e, \delta_1 \}$, and for $j \in [2, n]$, 
$$
w_j \in \left\{ \begin{array}{ll}
\{ w_{j-1}, w_{j-1} \delta_j \},  & \text{ if } \epsilon_j = -1    \\
\{ w_{j-1}, \delta_j w_{j-1} \},  & \text{ if } \epsilon_j = 1.    
\end{array} \right.
$$
\end{lem}
\begin{proof}
Let $x \in Z^{\sigma}_{\bfu, \bfv}$, and $\bfw = \Phi_n(x)$. For $j \in [1,n-1]$, we write $x_j = \rho_{j,n}(x) \in DF_j$ and set $x_n = x$. If $\epsilon_1 = -1$, then $x_1 \in \calO^{e,e} \cup \calO^{{s_1}, e}$. Similarly, if $\epsilon_1 = 1$, $x_1 \in \calO^{e,e} \cup \calO^{e, s_{l+1}}$. By Proposition \ref{w <- v^-1 ast u}, the assertion is true for $j = 1$. Thus assume $j \geqslant 2$, and write $(g.B, h.B_-) = \theta_{j-1}(x_{j-1})$, so that one has $h^{-1}g \in B_-w_{j-1}B$.   \\
\textbf{Case $\epsilon_j = -1$}: One has $\theta_j(x_j) = (gp.B, h.B_-)$, where $p \in P_{\delta_j}$. Thus $h^{-1}gp \in B_-w_jB$, and so 
$$
B_-w_jB \subset B_-w_{j-1}BP_{\delta_j} = B_-w_{j-1}B \cup B_-w_{j-1}B\delta_jB.  
$$
One knows by \cite[Appendix A]{HL} that 
$$
B_-w_{j-1}B\delta_jB \subset B_-w_{j-1}\delta_jB \cup B_-w_{j-1}B, 
$$
which proves the assertion.   \\
\textbf{Case $\epsilon_j = 1$}: One has $\theta_j(x_j) = (g.B, hp.B_-)$, where $p \in P_{-\delta_j}$. Thus $p^{-1}h^{-1}g \in B_-w_jB$, and 
$$
B_-w_jB \subset P_{-\delta_j}B_-w_{j-1}B = B_-w_{j-1}B \cup B_-\delta_jB_-w_{j_1}B.
$$
Again, by using \cite[Appendix A]{HL} one concludes that $w_j \in \{ w_{j-1}, \delta_j w_{j-1} \}$. 
\end{proof}

Let $W^{\bfu, \bfv, \sigma}$ be the elements of $W^n$ satisfying the condition in Lemma \ref{w_1, ..., w_n}. Thus $\Phi_n(Z^{\sigma}_{\bfu, \bfv}) \subset W^{\bfu, \bfv, \sigma}$. We will show in Proposition \ref{C^gamma = tildeC^gamma} that $\Phi_n(Z^{\sigma}_{\bfu, \bfv}) = W^{\bfu, \bfv, \sigma}$.  Lemma \ref{w_1, ..., w_n} motivates the following definition. 

\begin{defn} \label{shuffled subexpressions}
Let $\bfu = (s_1, ..., s_l)$ be a sequence of simple reflections. We say that a sequence $\gamma = (\gamma_1, ..., \gamma_l)$ is a \textit{subexpression of $\bfu$} if $\gamma_j \in \{ e, s_j \}$ for all $j \in [1,l]$. Denote by $\Upsilon_{\bfu}$ the set of all subexpressions of $\bfu$. Let $\bfv = (s_{l+1}, ..., s_n)$ be another sequence, and let $\sigma \in S_{l,n}$.  We say that $\gamma = (\gamma_1, ..., \gamma_n)$ is a \textit{subexpression of $(\bfu, \bfv)$ shuffled by $\sigma$}, or a \textit{$\sigma$-shuffled subexpression of $(\bfu, \bfv)$}, if $\gamma \in \Upsilon_{\sigma(\bfu, \bfv)}$. 
\end{defn}

Fix $(\bfu, \bfv, \sigma)$ as in Definition \ref{sigma bfu bfv}, and let $\gamma \in \Upsilon_{\sigma(\bfu, \bfv)}$. Construct the sequence $(\gamma^0, \gamma^1..., \gamma^n)$ of Weyl group elements as follows. Let $\gamma^0 = e$, and for any $j \in [1,n]$, define 
\begin{equation} \label{defn of gamma^j}
\gamma^j = \left\{ \begin{array}{ll}
\gamma^{j-1} \gamma_j, & \text{ if } \epsilon_j = -1  \\
\gamma_j \gamma^{j-1}, & \text{ if } \epsilon_j = 1.  
\end{array} \right.
\end{equation}
In other words, $(\gamma^1, ..., \gamma^n) \in W^{\bfu, \bfv, \sigma}$. Conversely, it is easily seen that any sequence $(w_0, ..., w_n)$ of elements in $W$ satisfying $w_0 = e$ and $(w_1, ..., w_n) \in W^{\bfu, \bfv, \sigma}$ uniquely defines a $\gamma \in \Upsilon_{\sigma(\bfu, \bfv)}$ such that $(\gamma^1, ..., \gamma^n) = (w_1, ..., w_n)$. Hence the map 
$$
\Upsilon_{\sigma(\bfu, \bfv)} \rightarrow W^{\bfu, \bfv, \sigma}, \;\; \gamma = (\gamma_1, ..., \gamma_n) \mapsto (\gamma^1, ..., \gamma^n)
$$
is a bijection.

\begin{exa}
Let $(\bfu, \bfv, \sigma)$ as in Example \ref{example1}. Then the sequences 
$$
\gamma = (s_1, s_4, e, s_2, e) \text{ and } \eta = (e, s_4, e, e, s_3)
$$
are examples of $\sigma$-shuffled subexpressions of $(\bfu, \bfv)$. One has 
$$
\begin{array}{lllllllllll}
(\gamma^1, & \gamma^2, & \gamma^3, & \gamma^4, & \gamma^5) & = & (s_1, & s_4s_1, & s_4s_1, & s_4s_1s_2, & s_4s_1s_2)   \\
(\eta^1, & \eta^2, & \eta^3, & \eta^4, &\eta^5) & = & (e, & s_4, & s_4, & s_4, & s_4s_3). 
\end{array}
$$
\end{exa}

\begin{defn} \label{C^gamma1}
For any $\gamma \in \Upsilon_{\sigma(\bfu, \bfv)}$, let 
$$
C^{\gamma} = \Phi_n^{-1}(\gamma^1, ..., \gamma^n) \cap Z^{\sigma}_{\bfu, \bfv} \;\;\; \text{ and } \;\;\;
C^{\gamma}_{\sigma} = \varphi^{\sigma}(C^{\gamma}) \subset Z_{\bfu, \bfv}.
$$ 
\end{defn}

One thus has the two disjoint unions 
\begin{equation} \label{disjoint unions}
Z^{\sigma}_{\bfu, \bfv} = \bigsqcup_{\gamma \in \Upsilon_{\sigma(\bfu, \bfv)}} C^{\gamma}  \;\;\; \text{ and } \;\;\; Z_{\bfu, \bfv} = \bigsqcup_{\gamma \in \Upsilon_{\sigma(\bfu, \bfv)}} C^{\gamma}_{\sigma}.
\end{equation}

\subsection{Affine charts and coordinates on $Z^{\sigma}_{\bfu, \bfv}$ }

We introduce an atlas of $2^n$ affine charts on $Z^{\sigma}_{\bfu, \bfv}$ parametrized by $\Upsilon_{\sigma(\bfu, \bfv)}$. As a first application, we show that all the $C^{\gamma}$'s are isomorphic to affine spaces. \\
\indent
Fix $(\bfu, \bfv, \sigma)$ as in Definition \ref{sigma bfu bfv} and let $\gamma \in \Upsilon_{\sigma(\bfu, \bfv)}$. We fix representatives in $N_G(H)$ of the Weyl group elements $\gamma^j$. Set $\wt{\gamma^0} = e$, and for $j \in [1,n]$, let 
\begin{equation} \label{gamma tilde}
\wt{\gamma^j} = \left\{ \begin{array}{ll}
\wt{\gamma^{j-1}} \bar{\gamma_j}, &  \text{ if } \epsilon_j = -1  \\
\bar{\gamma}_j \wt{\gamma^{j-1}}, &  \text{ if } \epsilon_j = 1.
\end{array} \right.  
\end{equation}
Recall from (\ref{delta}) that $\sigma(\bfu, \bfv) =  (\delta_1, ... ,\delta_n)$. Denote by $\alpha_j$ the simple root such that $\delta_j = s_{\alpha_j}$. For $j \in [1,n]$ and $z \in \IC$, let 
\begin{equation} \label{u_gamma, j(z)}
u_{\gamma, j}(z) = (p_{\gamma, j}(z), q_{\gamma, j}(z)) \in P^{\sigma}_j, 
\end{equation}
where
\begin{equation} \label{p_gamma, j(z)}
p_{\gamma, j}(z) = \left\{ \begin{array}{ll}
x_{-\gamma_j \alpha_j}(z) \bar{\gamma}_j, & \text{ if } \epsilon_j = -1   \\
\text{[} \wt{\gamma^{j-1}}^{-1} x_{\gamma_j \alpha_j}(z) \wt{\gamma^{j-1}} \text{]}_+, & \text{ if } \epsilon_j = 1,
\end{array} \right.
\end{equation}
and 
\begin{equation} \label{q_gamma, j(z)}
q_{\gamma, j}(z) = \left\{ \begin{array}{ll}
\text{[} \wt{\gamma^{j-1}}x_{-\gamma_j \alpha_j}(z) \wt{\gamma^{j-1}}^{-1} \text{]}_-, & \text{ if } \epsilon_j = -1   \\
x_{\gamma_j \alpha_j}(z) \bar{\gamma}_j^{-1}, & \text{ if } \epsilon_j = 1. 
\end{array} \right.
\end{equation}
Define 
\begin{equation} \label{calO^gamma}
\calO^{\gamma}  = 
\{[u_{\gamma, 1}(z_1), ..., u_{\gamma, n}(z_n)]_{Z^{\sigma}_{\bfu, \bfv}} \mid z = (z_1, ..., z_n) \in \IC^n \}.
\end{equation}
The fact that the map
\begin{equation}  \label{u_gamma z}
u_{\gamma}: \IC^n \rightarrow \calO^{\gamma}, \;\; 
z \mapsto u_{\gamma}(z) = [u_{\gamma, 1}(z_1), ..., u_{\gamma, n}(z_n)]_{Z^{\sigma}_{\bfu, \bfv}}
\end{equation}
is an isomorphism is a consequence of applying inductively the following Lemma \ref{why coordinates}. 

\begin{lem} \label{why coordinates}
Let $\alpha \in \Gamma$, and let $B \times B_-$ act on $P_{s_{\alpha}} \times B_-$ by right multiplication. For any $\dot{w} \in N_G(H)$, the sets 
$$
\{ (x_{\alpha}(z)\bar{s}_{\alpha}, [ \dot{w}x_{\alpha}(z) \dot{w}^{-1}]_-) \mid z \in \IC \}, \text{ and } 
\{ (x_{-\alpha}(z), [ \dot{w}x_{-\alpha}(z) \dot{w}^{-1}]_-) \mid z \in \IC \}
$$
intersect each $(B \times B_-)$-orbits in exactly one point. Consequently, the two maps 
\begin{align*}
(x_{\alpha}(z)\bar{s}_{\alpha}.B,\;  B_-) & \mapsto (x_{\alpha}(z)\bar{s}_{\alpha}, \; [ \dot{w}x_{\alpha}(z) \dot{w}^{-1}]_-), \\  
(x_{-\alpha}(z).B, \; B_-) & \mapsto (x_{-\alpha}(z), \; [ \dot{w}x_{-\alpha}(z) \dot{w}^{-1}]_-), \; z \in \IC 
\end{align*}
are sections over open subsets of the fiber bundle 
$$
P_{s_{\alpha}} \times B_- \rightarrow (P_{s_{\alpha}} \times B_-)/(B \times B_-).
$$
\end{lem}
\begin{proof}
Suppose that $(x_{\alpha}(z)\bar{s}_{\alpha}, [ \dot{w}x_{\alpha}(z) \dot{w}^{-1}]_-)(b, b_-) = (x_{\alpha}(z')\bar{s}_{\alpha}, [ \dot{w}x_{\alpha}(z') \dot{w}^{-1}]_-)$ for some $z, z' \in \IC$, $b \in B$, $b_- \in B_-$. Looking at the first component, one has $b = e$ and $z = z'$. Thus $b_- = e$. Proceed similarly for the other case.
\end{proof}

Moreover, one sees that $Z^{\sigma}_{\bfu, \bfv}$ is covered by the open subsets $\calO^{\gamma}$, $\gamma \in \Upsilon_{\sigma(\bfu, \bfv)}$. 

\begin{exa} \label{example3}
Let $G = SL(3, \IC)$, which has root system $A_2$. We follow the Bourbaki convention for the labelling of simple roots in simple roots systems and denote the $i$'th simple reflection by the german letter $\gs_i$ (to distinguish it from $s_i$, which is the $i$'th element of the sequence $\bfu$ or $\bfv$) and the $i$'th simple root by $\alpha_i$. Let $\bfu = \bfv = (\gs_1, \gs_2)$ and $\sigma \in S_{2,4}$ be the transposition $(34)$, so that $\epsilon = (-1, 1, -1, 1)$. Let $\gamma = (e, \gs_1, e, e) \in \Upsilon_{\sigma(\bfu, \bfv)}$ and $z = (z_1, z_2, z_3, z_4) \in \IC^4$. Then 
$$
\begin{array}{ll}
u_{\gamma, 1}(z_1) =  (x_{-\alpha_1}(z_1), x_{-\alpha_1}(z_1)),       &   u_{\gamma, 2}(z_2) =  (e , x_{-\alpha_1}(z_2)\bar{\gs}_1^{-1}),    \\
u_{\gamma, 3}(z_3) =  (x_{-\alpha_2}(z_3), \bar{\gs}_1x_{-\alpha_2}(z_3)\bar{\gs}_1^{-1}),   & u_{\gamma, 4}(z_4) = (\bar{\gs}_1^{-1}x_{\alpha_2}(z_4)\bar{\gs}_1 , x_{\alpha_2}(z_4)),
\end{array}
$$ 
and 
$$
\varphi^{\sigma}(u_{\gamma}(z_1, z_2, z_3, z_4)) = \left( [x_{-\alpha_1}(z_1), x_{-\alpha_2}(z_3)]_{Z_{\bfu}}, [x_{-\alpha_1}(z_1+z_2)\bar{\gs}_1, x_{\alpha_2}(z_4)]_{Z_{\bfv}} \right) \in Z_{\bfu, \bfv}.
$$
For $\eta = (e, e, \gs_2, \gs_2)$ one has 
$$
\begin{array}{ll}
u_{\eta, 1}(z_1) = (x_{-\alpha_1}(z_1), x_{-\alpha_1}(z_1)), & u_{\eta, 2}(z_2) = (x_{\alpha_1}(z_2), x_{\alpha_1}(z_2)),    \\
u_{\eta, 3}(z_3) = (x_{\alpha_2}(z_3)\bar{\gs}_2, e), & u_{\eta, 4}(z_4) = (x_{\alpha_2}(-z_4), x_{-\alpha_2}(z_4)\bar{\gs}_2^{-1}),
\end{array}
$$
and 
$$
\varphi^{\sigma}(u_{\eta}(z_1, z_2, z_3, z_4)) = \left( [x_{-\alpha_1}(z_1), x_{\alpha_2}(z_3)\bar{\gs}_2]_{Z_{\bfu}}, [x_{-\alpha_1}(z_1)x_{\alpha_1}(z_2), x_{-\alpha_2}(z_4)\bar{\gs}_2^{-1}]_{Z_{-\bfv}} \right).
$$
\end{exa}

\begin{defn} \label{J gamma}
For any $\gamma \in \Upsilon_{\sigma(\bfu, \bfv)}$, let 
$$
J(\gamma) = \{ j \in [1,n] \mid  (\gamma^j)^{-\epsilon_j} \alpha_j < 0 \}.
$$
\end{defn}

\begin{lem} \label{key lemma}
Let $\gamma \in \Upsilon_{\sigma(\bfu, \bfv)}$. Let $(z_1, ..., z_n) \in \IC^n$ be such that $z_j = 0$ if $j \in J(\gamma)$. Then for any $j \in [1,n]$, 
$$
q_{\gamma, j}(z_j)^{-1} \cdots q_{\gamma, 1}(z_1)^{-1}p_{\gamma, 1}(z_1) \cdots p_{\gamma, j}(z_j) = \wt{\gamma^j}.
$$
\end{lem}
\begin{proof}
Let $z = (z_1, ..., z_n) \in \IC^n$ with $z_j = 0$ if $j \in J(\gamma)$. Denote by $F_{\gamma, j}(z)$ the product on the left hand side, and set $F_{\gamma, 0} = e \in G$. Thus $F_{\gamma, 0}(z) = \wt{\gamma^0}$. So assume now that $j \in [1,n]$ and that $F_{\gamma, j-1}(z) = \wt{\gamma^{j-1}}$. \\
\textbf{Case $\epsilon_j = -1$}: Using the induction hypothesis, one has 
\begin{align*}
F_{\gamma, j}(z)& = q_{\gamma, j}(z_j)^{-1} F_{\gamma, j-1}(z) p_{\gamma, j}(z_j)     \\
 & = [\wt{\gamma^{j-1}}x_{-\gamma_j \alpha_j}(-z_j) \wt{\gamma^{j-1}}^{-1}]_- \wt{\gamma^{j-1}} x_{-\gamma_j \alpha_j}(z_j) \bar{\gamma}_j
\end{align*}
If $j \notin J(\gamma)$, then $\gamma^j \alpha_j = \gamma^{j-1} \gamma_j \alpha_j > 0$, and so $F_{\gamma, j}(z) = \wt{\gamma^j}$. If $j \in J(\gamma)$, then $z_j = 0$, so $F_{\gamma, j}(z) = \wt{\gamma^j}$.   \\
\textbf{Case $\epsilon_j = 1$}: One has 
\begin{align*}
F_{\gamma, j}(z)& = q_{\gamma, j}(z_j)^{-1} F_{\gamma, j-1}(z) p_{\gamma, j}(z_j)    \\
 & = \bar{\gamma}_j x_{\gamma_j \alpha_j}(-z_j)\wt{\gamma^{j-1}} [\wt{\gamma^{j-1}}^{-1} x_{\gamma_j \alpha_j}(z) \wt{\gamma^{j-1}}]_+
\end{align*}
If $j \notin J(\gamma)$, then $(\gamma^j)^{-1} \alpha_j = (\gamma^{j-1})^{-1} \gamma_j \alpha_j > 0$, and so $F_{\gamma, j}(z) = \wt{\gamma^j}$. If $j \in J(\gamma)$, then $z_j = 0$, so $F_{\gamma, j}(z) = \wt{\gamma^j}$. 
\end{proof}

The following Proposition \ref{C^gamma = tildeC^gamma} shows that each $C^{\gamma}$ in the combinatorially defined decomposition of $Z^{\sigma}_{\bfu, \bfv}$ in (\ref{disjoint unions}) is isomorphic to an affine space.

\begin{pro} \label{C^gamma = tildeC^gamma}
For any $\gamma \in \Upsilon_{\sigma(\bfu, \bfv)}$, one has 
$$
C^{\gamma} = \{u_{\gamma}(z) \in \calO^{\gamma} \mid z_j = 0, \text{ if } j \in J(\gamma) \}. 
$$ 
\end{pro}
\begin{proof}
Recall that $C^{\gamma}$ has been introduced in Definition \ref{C^gamma1}. Denote by $\tilde{C}^{\gamma}$ the set on the right hand side. Recall that an element 
$$
[(g_1,h_1), ..., (g_n,h_n)]_{DF_n}
$$
lies in $C^{\gamma}$ if and only if for every $j \in [1,n]$, $(g_1\cdots g_j.B, h_1\cdots h_j.B_-) \in G_{\diag}(\gamma^j)$, which is equivalent to 
\begin{equation} \label{condition for C^gamma}
(h_1 \cdots h_j)^{-1}g_1\cdots g_j \in B_-\gamma^jB, \;\; j \in [1,n]. 
\end{equation}
By Lemma \ref{key lemma}, one has $\tilde{C}^{\gamma} \subset C^{\gamma}$, so we need to prove 
\begin{equation} \label{C^gamma subset tildeC^gamma}
C^{\gamma} \subset \tilde{C}^{\gamma}.
\end{equation}
For $k \in [1,n]$, let 
$$
Z_k = P^{\sigma}_1 \times_{B \times B_-} \cdots  \times_{B \times B_-}  P^{\sigma}_k/(B \times B_-) \subset DF_k.
$$
Define the following two subsets of $Z_k$. 
\begin{align*}
C^{(\gamma_1, ..., \gamma_k)} & = \{[(g_1,h_1), ..., (g_k,h_k)]_{DF_k} \in Z_k \mid \text{(\ref{condition for C^gamma}) holds for } j \in [1,k] \}    \\
\tilde{C}^{(\gamma_1, ..., \gamma_k)} & = \{ [u_{\gamma,1}(z_1), ..., u_{\gamma, k}(z_k)]_{DF_k} \in Z_k \mid z_i = 0, \text{ if } i \in [1,k] \cap J(\gamma) \}. 
\end{align*}
Fix $x = [(g_1,h_1), ..., (g_n,h_n)]_{DF_n} \in C^{\gamma}$, i.e $x$ satisfies (\ref{condition for C^gamma}). Write $x_k = \rho_{k,n}(x)$ for $k \in [1,n-1]$ and set $x_n = x$. Then $x_k \in C^{(\gamma_1, ..., \gamma_k)}$ for every $k \in [1,n]$. We show by induction on $k$ that $x_k \in \tilde{C}^{(\gamma_1, ..., \gamma_k)}$ for every $k \in [1,n]$. Then (\ref{C^gamma subset tildeC^gamma}) is the statement for $k = n$. \\
\indent 
Assume that $\epsilon_1 = -1$, the case $\epsilon_1 = 1$ being similar. So $h_1 \in B_-$ and $g_1 \in P_{\delta_1}$. Then $g_1 \in B_- \gamma_1B$. Hence one can write $g_1 = x_{-\gamma_1 \alpha_1}(z_1)\bar{\gamma}_1b_1$ with $b_1 \in B$ and $z_1 = 0$ if $\gamma_1 = \delta_1$, that is if $1 \in J(\gamma)$. Thus 
$$
(g_1.B, h_1.B_-) = (x_{-\gamma_1 \alpha_1}(z_1)\bar{\gamma}_1.B, [x_{-\gamma_1 \alpha_1}(z_1)]_-.B_-) \in \tilde{C}^{(\gamma_1)}. 
$$
Let now $k \in [1, n-1]$, and assume that $C^{(\gamma_1, ..., \gamma_k)} \subset \tilde{C}^{(\gamma_1, ..., \gamma_k)}$. Once again the two cases of $\epsilon_{k+1} = 1$ or $\epsilon_{k+1} = -1$ can be treated similarly, so we can assume that $\epsilon_{k+1} = -1$. So $\gamma^{k+1} = \gamma^k \gamma_{k+1}$. By the induction hypothesis, $x_k \in \tilde{C}^{(\gamma_1, ..., \gamma_k)}$, thus there exist $g_{k+1}' \in P_{\delta_{k+1}}$, $h_{k+1}' \in B_-$ and $(z_1, ..., z_k) \in \IC^k$ with $z_j = 0$ if $j \in [1,k] \cap J(\gamma)$, such that 
$$
x_{k+1} = [u_{\gamma,1}(z_1), ..., u_{\gamma, k}(z_k), (g_{k+1}' , h_{k+1}' )]_{DF_{k+1}}. 
$$
Since $x_{k+1} \in C^{(\gamma_1, ..., \gamma_k)}$, we have 
$$
h_{k+1}'^{-1}q_{\gamma, k}(z_k)^{-1} \cdots q_{\gamma, 1}(z_1)^{-1}p_{\gamma, 1}(z_1) \cdots p_{\gamma, k}(z_k)g_{k+1}'  = h_{k+1}'^{-1}\wt{\gamma^k} g_{k+1}' \in B_-\gamma^{k+1}B. 
$$
Hence $\wt{\gamma^k}g_{k+1}' \in B_-\gamma^{k+1}B$. If $\gamma_{k+1} = \delta_{k+1}$, then $\gamma^{k+1} = \gamma^k \delta_{k+1} \neq \gamma^k$ and one cannot have $g_{k+1}' \in B$. Similarly, if $\gamma_{k+1} = e$, then one cannot have $g_{k+1}' \in \delta_{k+1}B$. Thus $g_{k+1}'$ can be written as $x_{-\gamma_{k+1} \alpha_{k+1}}(z_{k+1})\bar{\gamma}_{k+1}b_{k+1}$ for some $z_{k+1} \in \IC$ and $b_{k+1} \in B$. Furthermore, if $k+1 \in J(\gamma)$, that is if $\gamma^k \gamma_{k+1}\alpha_{k+1} < 0$, then one must have $z_{k+1} = 0$. If $k+1 \notin J(\gamma)$, then $\wt{\gamma^k}g_{k+1}' \in B_-\gamma^{k+1}B$ for any value of $z_{k+1}$. Moreover one can write $h_{k+1}' = [\wt{\gamma^k}x_{-\gamma_{k+1} \alpha_{k+1}}(z_{k+1}) \wt{\gamma^k}^{-1}]_-b_{-(k+1)}$, for some $b_{-(k+1)} \in B_-$. Hence 
$$
x_{k+1} = [u_{\gamma,1}(z_1), ..., u_{\gamma, k}(z_k), u_{\gamma, k+1}(z_{k+1})]_{Z_{k+1}} \in \tilde{C}^{(\gamma_1, ..., \gamma_{k+1})}. 
$$
\end{proof}

\begin{exa}
Let $\bfu, \bfv, \sigma, \gamma$, and $\eta$ be as in Example \ref{example3}. Then $J(\gamma) = \{2\}$ and $J(\eta) = \{3\}$. So 
$$
C^{\gamma} = \{u_{\gamma}(z_1, 0, z_3, z_4) \mid z_i \in \IC\} \text{ and } C^{\eta} = \{u_{\eta}(z_1, z_2, 0, z_4) \mid z_i \in \IC\}. 
$$
\end{exa}

\subsection{The $H$-action on $\calO^{\gamma}$ in the $z$-coordinates}

Fix again $(\bfu, \bfv, \sigma)$ as in Definition \ref{shuffled subexpressions}. Let $H$ act diagonally on $Z_{\bfu, \bfv}$. This action corresponds, before applying $\varphi^{\sigma}$, to the action on $Z^{\sigma}_{\bfu, \bfv}$ given by 
$$
h \cdot [(g_1, h_1), ...(g_n, h_n)]_{Z^{\sigma}_{\bfu, \bfv}} = [(hg_1, hh_1), (g_2, h_2), ...(g_n, h_n)]_{Z^{\sigma}_{\bfu, \bfv}}, \; h \in H, (g_i, h_i) \in P^{\sigma}_i. 
$$
Let $\gamma \in \Upsilon_{\sigma(\bfu, \bfv)}$. We show in this section that $\calO^{\gamma}$ is invariant under the $H$-action, and compute the action of $H$ in the coordinates $(z_1, ..., z_n)$. For any $j \in [1,n]$, one can write $\gamma^j = (\gamma_{\bfv}^j)^{-1}\gamma_{\bfu}^j$, with
\begin{equation}  \label{ga bfu and ga bfv}
\gamma_{\bfv}^j = \prod_{1 \leqslant k \leqslant j, \epsilon_k =1} \gamma_k, \;\;\; \gamma_{\bfu}^j =  \prod_{1 \leqslant k \leqslant j, \epsilon_k =-1} \gamma_k, 
\end{equation}
where the index in both products is increasing. Set $\gamma_{\bfu}^0 = \gamma_{\bfv}^0 = e$. 
In particular, one can write $\gamma^n = \gamma_{\bfv}^{-1} \gamma_{\bfu}$, with $\gamma_{\bfv} = \gamma_{\bfv}^n$, $\gamma_{\bfu} = \gamma_{\bfu}^n$.

\begin{pro} \label{haction}
Let $\gamma \in \Upsilon_{\sigma(\bfu, \bfv)}$. Then $\calO^{\gamma}$ is invariant under $H$, and each coordinate function $z_j$, $j \in [1,n]$, is a weight function for $H$, with
$$
h \cdot z_j = \left\{ \begin{array}{ll}
h^{- \gamma_{\bfu}^j \alpha_j} z_j, & \epsilon_j = -1    \\
h^{ \gamma_{\bfv}^j \alpha_j} z_j, & \epsilon_j = 1. 
\end{array} \right. 
$$
\end{pro}
\begin{proof}
Let $h_1, h_2 \in H$. Let $\alpha \in \Gamma$, and $\delta \in \{e, s_{\alpha} \}$. For any $\dot{w} \in N_G(H)$ and $z \in \IC$, one has 
\begin{align}
(h_1, h_2)(x_{-\delta \alpha}(z)\bar{\delta}, [\dot{w}x_{-\delta \alpha}(z) \dot{w}^{-1}]_-) & =  \notag  \\
   &  (x_{-\delta \alpha}(h_1^{-\delta \alpha}z) \bar{\delta}, [\dot{w}x_{-\delta \alpha}(h_2^{-w\delta \alpha}z) \dot{w}^{-1}]_-)(h_1^{\delta}, h_2) \label{brother1}  \\
(h_1, h_2)([\dot{w}^{-1}x_{\delta \alpha}(z) \dot{w}]_+, x_{\delta \alpha}(z)\bar{\delta}^{-1}) & = \notag  \\
   &  ([\dot{w}^{-1}x_{\delta \alpha}(h_1^{w^{-1}\delta \alpha}z) \dot{w}]_+, x_{\delta \alpha}(h_2^{\delta \alpha}z) \bar{\delta}^{-1})(h_1, h_2^{\delta})  \label{brother2}   
\end{align}
Let $h \in H$ and $j \in [1,n]$. If $\epsilon_j = -1$, using (\ref{brother1}), one gets 
\begin{align*}
(h^{\gamma_{\bfu}^{j-1}}, h^{\gamma_{\bfv}^{j-1}})(p_{\gamma, j}(z_j), q_{\gamma, j}(z_j)) & = (p_{\gamma, j}(h^{-\gamma_{\bfu}^j \alpha_j}z_j), q_{\gamma, j}(h^{-\gamma_{\bfv}^{j-1} \gamma^{j-1} \gamma_j \alpha_j}z_j))(h^{\gamma_{\bfu}^j}, h^{\gamma_{\bfv}^{j-1}})   \\
  & = (p_{\gamma, j}(h^{-\gamma_{\bfu}^j \alpha_j}z_j), q_{\gamma, j}(h^{-\gamma_{\bfu}^j \alpha_j}z_j))(h^{\gamma_{\bfu}^j}, h^{\gamma_{\bfv}^j}). 
\end{align*}
The same can be proved using (\ref{brother2}) if $\epsilon_j = 1$. Thus an induction on $j$ yields the result. 
\end{proof}

\begin{defn} \label{weight}
We denote by $\nu_{ \gamma, j}$ the weight of the coordinate $z_j$ for the action of $H$. That is 
$$
\nu_{\gamma, j} =  \left\{ \begin{array}{ll}
- \gamma_{\bfu}^j \alpha_j, & \text{ if } \epsilon_j = -1    \\
\gamma_{\bfv}^j \alpha_j, & \text{ if } \epsilon_j = 1. 
\end{array} \right. 
$$
\end{defn}

\subsection{Positive and distinguished $\sigma$-shuffled subexpressions}

We now prove a few facts on the combinatorics of shuffled subexpressions that will be useful later. In particular, we introduce positive and distinguished shuffled subexpressions. \\
\indent
Fix $(\bfu, \bfv, \sigma)$ as in Definition \ref{sigma bfu bfv}, and let $\sigma(\bfu, \bfv) = (\delta_1, ..., \delta_n)$ be as in (\ref{delta}). 

\begin{defn} \label{defn of positivity}
Let $\gamma \in \Upsilon_{\bfu}$. We say that $\gamma$ is \textit{positive}, if 
$$
\gamma_1 \cdots \gamma_{j-1} = (\gamma_1 \cdots \gamma_j) \lhd s_j,
$$
for any $j \in [1,l]$. We write $\Upsilon_{\bfu}^+$ for the set of all positive subexpressions of $\bfu$. We say that $\gamma \in \Upsilon_{\sigma(\bfu, \bfv)}$ is \textit{$\sigma$-positive}, or that $\gamma$ is a \textit{positive $\sigma$-shuffled subexpression of $(\bfu, \bfv)$}, if for any $j \in [1,n]$, 
$$
\gamma^{j-1} = \left\{ \begin{array}{ll}
\gamma^j \lhd \delta_j, & \text{ if } \epsilon_j = -1   \\
\delta_j \rhd \gamma^j, & \text{ if } \epsilon_j = 1.  
\end{array} \right.
$$
Denote by $\Upsilon_{\bfu, \bfv, \sigma}^+$ the set of all positive $\sigma$-shuffled subexpressions of $(\bfu, \bfv)$.
\end{defn}

\begin{rem}
Recall that $\gamma^j \lhd \delta_j = \min \{\gamma^j, \gamma^j\delta_j \}$, and $\delta_j \rhd \gamma^j = \min \{\gamma^j, \delta_j \gamma^j\}$. So one sees that $\gamma$ is positive if and only if for every $j \in [1,n]$, 
$$
\left\{ \begin{array}{ll}
\gamma^{j-1} < \gamma^{j-1} \delta_j, & \text{ when } \epsilon_j = -1   \\ 
\gamma^{j-1} < \delta_j  \gamma^{j-1}, &  \text{ when } \epsilon_j = 1.
\end{array} \right.
$$
\end{rem}

The notion of positivity of subexpressions has previously been studied in \cite{MR, WY}. Our definition of positivity for shuffled subexpressions is a generalization of the notion of ``positive double subexpression" found in \cite{WY}. Indeed, if $\bfu$ and $\bfv$ are reduced, the two notions coincide. See Section 3.8 for more details.

\begin{exa}
In the root system $A_4$, let 
$$
\bfu = (\gs_3,\gs_2,\gs_3, \gs_2, \gs_3), \;\; \bfv = (\gs_4,\gs_2,\gs_1,\gs_4),
$$
and $\sigma$ be the $(5, 9)$-shuffle corresponding to 
$$
\begin{array}{lllllllllll}
\epsilon(\sigma) & = &(-1,         & -1,       & 1,         & -1,       & 1,         & 1,         & -1,       & -1,        & 1),
\end{array}
$$
so that 
$$
\begin{array}{lllllllllll}
\sigma(\bfu, \bfv) & = & (\gs_3, & \gs_2, & \gs_4, & \gs_3, & \gs_2, & \gs_1, & \gs_2, & \gs_3, & \gs_4).
\end{array}
$$
Consider the $\sigma$-shuffled subexpressions
$$
\gamma_1  = (e,e,e,\gs_3, e, e, \gs_2, \gs_3, \gs_4) \text{ and } \gamma_2  = (e,e,\gs_4,\gs_3,\gs_2,\gs_1,\gs_2,\gs_3,e).
$$
Then $\gamma_1$ is positive, but $\gamma_2$ is not. 
\end{exa}

Let 
\begin{equation} \label{ast prod uv}
u = s_1 \ast \cdots \ast s_l, \;\;\; v = s_{l+1} \ast \cdots \ast s_n. 
\end{equation}
We wish to show that elements $w \in W$ with $w \leqslant v^{-1} \ast u$ are in 1-1 correspondance with positive $\sigma$-shuffled subexpressions.

\begin{defn}
For $\gamma \in \Upsilon_{\sigma(\bfu, \bfv)}$, let 
\begin{align*}
\gamma_{[1,l]} & = (\gamma_{\sigma(1)}, ..., \gamma_{\sigma(l)}) \in \Upsilon_{\bfu},   \\
\gamma_{[l+1,n]} & = (\gamma_{\sigma(l+1)}, ..., \gamma_{\sigma(n)}) \in \Upsilon_{\bfv}.  
\end{align*}
\end{defn}

\begin{lem} \label{double pos sub}
For any $w \in W$ with $w \leqslant v^{-1} \ast u$, there is a unique $\gamma \in \Upsilon_{\bfu, \bfv, \sigma}^+$ such that $\gamma^n = w$. Moreover $\gamma$ has the following properties. \\
1) For any $j \in [1,n]$, $l(\gamma^j) = l(\gamma_{\bfv}^j) + l(\gamma_{\bfu}^j)$;    \\
2) $\gamma_{[1,l]}$ is a positive subexpression of $\bfu$, and $\gamma_{[l+1,n]}$ is a positive subexpression of $\bfv$. 
\end{lem}
\begin{proof}
Let $w \in W$ such that $w \leqslant v^{-1} \ast u$. Set $w_n = w$, and for any $j \in [1,n]$, let 
$$
w_{j-1} = \left\{ \begin{array}{ll}
w_j \lhd \delta_j, & \text{ if } \epsilon_j = -1   \\
\delta_j \rhd w_j, & \text{ if } \epsilon_j = 1.
\end{array} \right.
$$
Then, by \cite[Lemma A.3]{HL}, one has $w_0 = v \rhd w \lhd u^{-1} = e$. Let $\gamma_j = w_{j-1}^{\epsilon_j}w_j^{-\epsilon_j}$. Then $\gamma = (\gamma_1, ..., \gamma_n)$ is a positive $\sigma$-shuffled subexpression, and is clearly the unique one satisfying $\gamma^n = w$. We now prove that 
\begin{equation} \label{length add up}
l(\gamma^j) = l(\gamma_{\bfv}^j) + l(\gamma_{\bfu}^j), \;\;\, j \in [1,n],
\end{equation}
for any $\gamma \in \Upsilon_{\bfu, \bfv, \sigma}^+$. 
One clearly has $l(\gamma^1) = l(\gamma_{\bfv}^1) + l(\gamma_{\bfu}^1)$. Thus suppose $j \geqslant 2$, and that $l(\gamma^{j-1}) = l(\gamma_{\bfv}^{j-1}) + l(\gamma_{\bfu}^{j-1})$. If $\gamma_j = e$, then (\ref{length add up}) holds. So assume $\gamma_j = \delta_j$. Suppose that $\epsilon_j = -1$, the other case being similar. Then $\gamma^j = \gamma^{j-1}\delta_j >  \gamma^{j-1}$ and one has $\gamma_{\bfu}^j = \gamma_{\bfu}^{j-1} \delta_j > \gamma_{\bfu}^{j-1}$. Thus 
$$
l(\gamma^j) =  l(\gamma_{\bfv}^{j-1}) + l(\gamma_{\bfu}^{j-1}) +1 =   l(\gamma_{\bfv}^j) + l(\gamma_{\bfu}^j). 
$$ 
This proves 1). Let now $j \in [1,n]$, and suppose that $\epsilon_j = 1$. One has $\delta_j \gamma^{j-1} > \gamma^{j-1}$. Since $l(\gamma^{j-1}) = l(\gamma_{\bfv}^{j-1}) + l(\gamma_{\bfu}^{j-1})$, this implies that $\delta_j (\gamma_{\bfv}^{j-1})^{-1} > (\gamma_{\bfv}^{j-1})^{-1}$. Or in other words, $(\gamma_{\bfv}^{j-1})^{-1} =  \delta_j  \rhd (\gamma_{\bfv}^j)^{-1}$, which is equivalent, by \cite[Lemma A3]{HL}, to 
$$
\gamma_{\bfv}^j \lhd \delta_j = \gamma_{\bfv}^{j-1}. 
$$
This shows that $(\gamma_{\sigma(l+1)}, ..., \gamma_{\sigma(n)})$ is a positive subexpression of $\bfv$. If $\epsilon_j = -1$, a similar computation shows that $(\gamma_{\sigma(1)}, ..., \gamma_{\sigma(l)})$ is a positive subexpression of $\bfu$.
\end{proof}

\begin{rem} \label{gamma tilde = gamma bar}
Recall the Weyl group representatives $\wt{\gamma^j}$ from (\ref{gamma tilde}). If $\gamma$ is positive, it follows from Lemma \ref{double pos sub} that $\wt{\gamma^j} = \overline{\gamma^j}$, for all $j \in [1,n]$.   
\end{rem}

\begin{defn}
If $w \leqslant u^{-1} \ast v$, we denote by $\gamma_w \in \Upsilon_{\bfu, \bfv, \sigma}^+$ the unique positive $\sigma$-shuffled subexpression satisfying $\gamma_w^n = w$. 
\end{defn}

On the other hand, for any $\gamma \in \Upsilon_{\sigma(\bfu, \bfv)}$, one has $\gamma^n = \gamma_{\bfv}^{-1} \gamma_{\bfu} \leqslant v^{-1} \ast u$. Thus we have shown the following. 

\begin{cor}
The map 
$$
\{w \in W \mid w \leqslant v^{-1} \ast u \} \rightarrow \Upsilon_{\bfu, \bfv, \sigma}^+, \;\; w \mapsto \gamma_w
$$
is a bijection. The inverse is given by $\gamma \mapsto \gamma^n$. 
\end{cor}

\begin{lem} \label{dim of C^gamma}
For any $\gamma \in \Upsilon_{\sigma(\bfu, \bfv)}$, one has $l(\gamma^n) \leqslant |J(\gamma)|$, with equality if and only if $\gamma$ is positive.
\end{lem}
\begin{proof}
We proceed by induction. The result is obvious when $n=1$, so let $n \geqslant 2$. We assume that $\epsilon_n = 1$, the case $\epsilon_n = -1$ being treated similarly. Then one has $\sigma(n) = n$. 
Let $\bfv' = (s_{l+1}, ..., s_{n-1})$, and $\gamma' = (\gamma_1, ..., \gamma_{n-1})$. Then $\gamma' \in \Upsilon_{\sigma'(\bfu, \bfv')}$, where $\sigma'(j) = \sigma(j)$, for all $j \in [1,n-1]$. By the induction hypothesis, the result holds for $\gamma'$. Recall that we denote by $\alpha_n$ the simple root such that $\delta_n = s_{\alpha_n}$. \\
\textbf{Case 1:} $(\gamma^n)^{-1} \alpha_n > 0$ and $\gamma_n = e$. \\
Then $n \notin J(\gamma)$, and $|J(\gamma)| = |J(\gamma')| \geqslant l(\gamma^{n-1}) = l(\gamma^n)$. Since $(\gamma^{n-1})^{-1} \alpha_n > 0$ one has $\delta_n \gamma^{n-1} > \gamma^{n-1}$. Hence $\gamma$ is positive if and only if $\gamma'$ is positive. So we are done.  \\
\textbf{Case 2:} $(\gamma^n)^{-1} \alpha_n > 0$ and $\gamma_n = \delta_n$. \\
Then $n \notin J(\gamma)$, and $|J(\gamma)| = |J(\gamma')| \geqslant l(\gamma^{n-1}) = l(\gamma^n) +1$. Since $(\gamma^{n-1})^{-1} \alpha_n < 0$ one has $\delta_n \gamma^{n-1} < \gamma^{n-1}$, so $\gamma$ is not positive.  \\
\textbf{Case 3:} $(\gamma^n)^{-1} \alpha_n < 0$, and $\gamma_n = e$. \\
Then $n \in J(\gamma)$, and $|J(\gamma)| = |J(\gamma')| + 1\geqslant l(\gamma^{n-1}) +1 = l(\gamma^n)+1$. Since $(\gamma^{n-1})^{-1} \alpha_n < 0$, $\gamma$ is not positive.   \\
\textbf{Case 4:} $(\gamma^n)^{-1} \alpha_n < 0$ and $\gamma_n = \delta_n$. \\
Then $n \in J(\gamma)$, and $|J(\gamma)| = |J(\gamma')| +1 \geqslant l(\gamma^{n-1}) +1 = l(\gamma^n)$. Since $(\gamma^{n-1})^{-1} \alpha_n > 0$, $\gamma$ is positive if and only if $\gamma'$ is positive. So we are done. 
\end{proof}

\begin{rem} \label{pos vs disting}
If $\gamma \in \Upsilon_{\bfu, \bfv, \sigma}^+$, then $j \in J(\gamma)$ if and only if $\gamma_j = \delta_j$. Indeed, $j \in J(\gamma)$ if and only if  $\delta_j (\gamma^j)^{\epsilon_j} < (\gamma^j)^{\epsilon_j}$. But since $\gamma$ is positive, this happens exactly when $\gamma_j = \delta_j$. 
\end{rem}

\begin{defn} \label{distinguished subexpr}
Let $\gamma \in \Upsilon_{\sigma(\bfu, \bfv)}$. We say that $\gamma$ is \textit{distinguished}, if $C^{\gamma} \cap \calO^{\sigma(\bfu, \bfv)}$ is non-empty. Denote by $\Upsilon_{\bfu, \bfv, \sigma}^d$ the set of all distinguished $\sigma$-shuffled subexpressions. 
\end{defn}

The open subset $\calO^{\sigma(\bfu, \bfv)} \subset Z^{\sigma}_{\bfu, \bfv}$ has an intrinsic description. Indeed, for $j \in [1,n]$, let 
$$
B^{\sigma}_j = \left\{
\begin{array}{cc}
B\delta_jB \times B_-, &  \epsilon_j = -1   \\
B \times B_-\delta_jB_-, & \epsilon_j = 1.
\end{array} \right.
$$
Then $\calO^{\sigma(\bfu, \bfv)} = B^{\sigma}_1 \times_{B \times B_-} \cdots \times_{B \times B_-} B^{\sigma}_n/(B \times B_-)$. Note that we have $\calO^{\sigma(\bfu, \bfv)} = I^{\sigma}_{\bfu, \bfv}(\calO^{\bfu, \bfv})$, recall (\ref{calO bfu bfv}) and (\ref{I^sigma_bfu bfv}).

\begin{defn} \label{I gamma}
For any $\gamma \in \Upsilon_{\sigma(\bfu, \bfv)}$, let 
$$
I(\gamma) = \{j \in [1,n] \mid \gamma_j = \delta_j \}.
$$
\end{defn}

\begin{lem} \label{J subset I}
A $\sigma$-shuffled subexpression $\gamma \in \Upsilon_{\sigma(\bfu, \bfv)}$ is distinguished if and only if $J(\gamma) \subset I(\gamma)$. 
\end{lem}
\begin{proof}
Recall (\ref{u_gamma, j(z)}). Notice that if $j \in I(\gamma)$, then $u_{\gamma, j}(z_j) \in B^{\sigma}_j$ for any $z_j \in \IC$. And if $j \notin I(\gamma)$, then $u_{\gamma, j}(z_j) \in B^{\sigma}_j$ if and only if $z_j \neq 0$. \\
\indent
Suppose that $J(\gamma) \subset I(\gamma)$. Let $z = (z_1, ..., z_n) \in \IC^n$, with $z_j = 0$ if $j \in J(\gamma)$, and $z_j \neq 0$ if $j \notin J(\gamma)$. Then by Proposition \ref{C^gamma = tildeC^gamma}, $u_{\gamma}(z) \in C^{\gamma} \cap \calO^{\sigma(\bfu, \bfv)}$. Conversely, if $J(\gamma)$ is not contained in $I(\gamma)$, let $j \in J(\gamma)$ such that $j \notin I(\gamma)$. If $u_{\gamma}(z) \in C^{\gamma}$, by Proposition \ref{C^gamma = tildeC^gamma}, one must have $z_j = 0$. But then $u_{\gamma}(z) \notin \calO^{\sigma(\bfu, \bfv)}$. 
\end{proof}

\begin{rem}
By Remark \ref{pos vs disting}, any positive $\sigma$-shuffled subexpression $\gamma$ is distinguished. Conversely, a distinguished $\sigma$-shuffled subexpression is positive if and only if $J(\gamma) = I(\gamma)$. 
\end{rem}

\begin{exa}
Let $\bfu, \bfv, \sigma, \gamma$ and $\eta$ be as in Example \ref{example3}. Then $\gamma$ is positive and $\eta$ is distinguished, but not positive. 
\end{exa}

Let $\gamma \in \Upsilon_{\bfu, \bfv, \sigma}^d$. Then $u_{\gamma}(z) \in \calO^{\sigma(\bfu, \bfv)}$ if and only if $z_j \neq 0$ when $j \notin I(\gamma)$. Combining this with Proposition \ref{C^gamma = tildeC^gamma}, one obtains the following 

\begin{pro}  \label{O cap C^gamma}
Let $\gamma \in \Upsilon_{\bfu, \bfv, \sigma}^d$ and let $(z_1, ..., z_n) \in \IC^n$ with 
$$
z_j \in \left\{
\begin{array}{ll}
\{0\}, & \text{ if } j \in J(\gamma)   \\
\IC, & \text{ if } I(\gamma) \backslash J(\gamma)    \\
\IC^*, & \text{ if } j \notin I(\gamma).
\end{array} \right.
$$
Then the map 
$$
\IC^{| I(\gamma) \backslash J(\gamma) |} \times (\IC^*)^{|I(\gamma)^c|} \rightarrow C^{\gamma} \cap \calO^{\sigma(\bfu, \bfv)},
$$
where $I(\gamma)^c = [1,n] \backslash I(\gamma)$, sending $z$ to $u_{\gamma}(z)$ is an isomorphism. 
\end{pro}

\subsection{The subvarieties $A_w$}

Let $\bfu = (s_1, ..., s_l)$ and $\bfv = (s_{l+1}, ..., s_n)$ be two sequences of simple reflections. Consider the double flag variety $Z_{\bfu, \bfv}$, and recall the map $\theta_{\bfu, \bfv}$ from (\ref{theta_ bfu bfv}).   

\begin{defn} \label{defn of A_w}
For $w \in W$, let 
$$
A_w = \theta_{\bfu, \bfv}^{-1}(G_{\diag}(w)) \subset Z_{\bfu, \bfv},
$$
and if $\sigma$ is a $(l,n)$-shuffle, let 
$$
A_w^{\sigma} = I^{\sigma}_{\bfu, \bfv}(A_w) \subset Z^{\sigma}_{\bfu, \bfv}, 
$$
recall (\ref{I^sigma_bfu bfv}).
\end{defn}

 Let $u, v$ be as in (\ref{ast prod uv}), and let $w \in W$  with $w \leqslant v^{-1} \ast u$. Then for any $\gamma \in \Upsilon_{\sigma(\bfu, \bfv)}$ with $\gamma^n = w$, by definition of $C^{\gamma}$ one has $C^{\gamma} \subset A^{\sigma}_w$. \\
\indent
Conversely, suppose that $A_w$ is non-empty. Since $\theta_{\bfu, \bfv}(Z_{\bfu, \bfv}) = \overline{\calO^{u,v}}$, Proposition \ref{w <- v^-1 ast u} implies that $w \leqslant v^{-1} \ast u$. Thus one gets 

\begin{pro} \label{A_w non empty iif}
Let $w \in W$. Then $A_w$ is non-empty if and only if $w\leqslant v^{-1} \ast u$. Moreover, for any $(l,n)$-shuffle $\sigma$, one has 
$$
A^{\sigma}_w = \bigsqcup_{\gamma \in \Upsilon_{\sigma(\bfu, \bfv)} : \gamma^n = w} C^{\gamma}. 
$$
\end{pro}

\begin{pro}
For any $w \in W$ with $w \leqslant v^{-1} \ast u$, $A_w$ is a smooth, irreducible subvariety of $Z_{\bfu, \bfv}$ of codimension $l(w)$ in $Z_{\bfu, \bfv}$.
\end{pro}
\begin{proof}
The image of $\theta_{\bfu, \bfv}$ can be written as 
$$
\theta_{\bfu, \bfv}(Z_{\bfu, \bfv}) = \overline{\calO^{u,v}} = \bigsqcup_{u_1 \leqslant u, v_1 \leqslant v} \calO^{u_1, v_1}. 
$$
On the other hand, it is known by \cite[Theorem 1.4]{R} that $\calO^{u_1, v_1}$ and $G_{\diag}(w)$ intersect transversally, for any $u_1, v_1, w \in W$. Hence $\theta_{\bfu, \bfv}$ is transverse to $G_{\diag}(w)$, and so $A_w$ is smooth, and every component of $A_w$ has codimension 
$$
\dim_{\IC} Z_{\bfu, \bfv} - \dim_{\IC}A_w = \dim_{\IC} DF_1 - \dim_{\IC}G_{\diag}(w) = l(w).  
$$
Choose now any $\sigma \in S_{l,n}$, and consider $\gamma_w \in \Upsilon_{\bfu, \bfv, \sigma}^+$. By Lemma \ref{dim of C^gamma}, $C^{\gamma_w}$ has codimension $l(w)$ in $Z^{\sigma}_{\bfu, \bfv}$, and so is open in $A^{\sigma}_w$. Furthermore, by Proposition \ref{A_w non empty iif}, the complement of $C^{\gamma_w}$ in $A^{\sigma}_w$ is a finite union of subvarieties of lower dimension. Hence $C^{\gamma_w}$ is dense in $A^{\sigma}_w$, and so $A^{\sigma}_w$ is irreducible in the Zariski topology. 
\end{proof}

\begin{thm} \label{C^gamma open in A^sigma_w}
Let $w \in W$ with $w \leqslant v^{-1} \ast u$. Then 
$$
C^{\gamma_w} = \calO^{\gamma_w} \cap A^{\sigma}_w = \calO^{\gamma_w} \cap \overline{A^{\sigma}_w}
$$
where $\overline{A^{\sigma}_w}$ is the Zariski closure of $A^{\sigma}_w$. 
\end{thm}
\begin{proof}
One has the inclusions $C^{\gamma_w} \subseteq \calO^{\gamma_w} \cap A^{\sigma}_w \subseteq \calO^{\gamma_w} \cap \overline{A^{\sigma}_w}$. On the other hand, both $C^{\gamma_w}$ and $\calO^{\gamma_w} \cap \overline{A^{\sigma}_w}$ are closed, irreducible varieties of $\calO^{\gamma_w}$ of the same dimension. Hence they must be equal. 
\end{proof}

\subsection{Relation to the Webster-Yakimov decompositions} \label{rel with WY}

Let $u,v,w \in W$ with $w  \leqslant v^{-1} \ast u$, and consider $\calO^{u,v}_w \subset DF_1$. In \cite{WY}, Webster and Yakimov introduced decompositions of $\calO^{u,v}_w$ into subvarieties isomorphic to $\IC^{n_1} \times (\IC^*)^{n_2}$ for some integers $n_1, n_2$. We show in this section that the decomposition of $A^{\sigma}_w$ in Proposition \ref{A_w non empty iif} recovers the decompositions in \cite{WY} when $\bfu$ and $\bfv$ are reduced. \\

\indent
We start by recalling the notations used in \cite{WY}. The set of simple roots of $G$ is labelled by $[1,r]$, and $W \times W$ is considered as a Coxeter group with simple reflections $s_{-i}, s_i$, for $i \in [1,r]$. Let $(u,v) \in W \times W$, and let $\bfi = (s_{i_1}, ..., s_{i_n})$, with $i_k \in  [1,r] \cup [-r, -1]$, be a reduced expression for $(u,v)$. We let $\epsilon(k) =1$, if $i_k \in [1,r]$, and $\epsilon(k) = -1$, if $i_k \in [-r, -1]$. A sequence $\bfw = (w_{(0)}, ..., w_{(n)})$ of elements in $W$ is called a \textit{double subexpression of $\bfi$}, if $w_{(0)} = e$, and for $k \in [1,n]$, 
$$
w_{(k)} \in \left\{ \begin{array}{ll}
\{ w_{(k-1)}, w_{(k-1)}s_{|i_k|} \}, & \epsilon(k) = 1,  \\
\{ w_{(k-1)}, s_{|i_k|}w_{(k-1)} \}, & \epsilon(k) = -1.
\end{array} \right.
$$
Webster and Yakimov call $\bfw$ \textit{double distinguished} if $w_{(k)}^{\epsilon(k)} = w_{(k-1)}^{\epsilon(k)}s_{|i_k|}$ for all $k \in [1,n]$ such that $w_{(k-1)}^{\epsilon(k)}s_{|i_k|} < w_{(k-1)}^{\epsilon(k)}$. For each double subexpression $\bfw$, we let 
$$
\begin{array}{lllll}
J_{\bfw}^0 & \text{be the set of indices} & k \in [1,n] & \text{such that} & w_{(k-1)} = w_{(k)},   \\
J_{\bfw}^+ & \text{be the set of indices} & k \in [1,n] & \text{such that} & w_{(k-1)} < w_{(k)},     \\
J_{\bfw}^- & \text{be the set of indices} & k \in [1,n] & \text{such that} & w_{(k-1)} > w_{(k)} .
\end{array}
$$
A double distinguished subexpression is called \textit{positive}, if $J_{\bfw}^-$ is empty. \\
\indent 
We translate now the notations of \cite{WY} into ours. Let 
$$
\bfu = (s_{-i_{k_1}}, ..., s_{-i_{k_l}}), \;\;\;   \bfv = (s_{i_{m_1}}, ..., s_{i_{m_{n-l}}}),
$$
where $\{ k_1 < ...< k_l \}$ is the set of indices in $[1,n]$ with $\epsilon({k_j}) = -1$, $j \in [1,l]$, and $\{ {m_1} < ... < {m_{n-l}} \}$ the set of indices in $[1,n]$ with $\epsilon({m_t}) = 1$, $t \in [1, n-l]$. Then $\bfu$ and $\bfv$ are reduced expressions. Let $\sigma$ be the unique $(l,n)$-shuffle such that $\sigma(\bfu, \bfv) = (s_{|i_1|}, ..., s_{|i_n|})$. It is clear that $\bfw$ is a double subexpression of $\bfi$ if and only if there exists $\gamma \in \Upsilon_{\sigma(\bfu, \bfv)}$ such that $w_{(k)}^{-1} = \gamma^k$, for any $k \in [1,n]$. Moreover, $\bfw$ is double distinguished if and only if $\gamma$ is distinguished, and in this case 
$$
\begin{array}{lcr}
J_{\bfw}^0 = [1,n] \backslash I(\gamma), & J_{\bfw}^+ = J(\gamma), & J_{\bfw}^- = I(\gamma) \backslash J(\gamma). 
\end{array}
$$
Then $\bfw$ is positive if and only if $\gamma$ is positive. \\
\indent
Since $\bfu$ and $\bfv$ are reduced, the restriction $\theta_n \mid_{\calO^{\sigma(\bfu, \bfv)}}: \calO^{\sigma(\bfu, \bfv)} \rightarrow \calO^{u,v}$ is an isomorphism, recall (\ref{teta_n}). One has 
$$
\calO^{\sigma(\bfu, \bfv)} = \bigsqcup_{\gamma \in \Upsilon_{\bfu, \bfv, \sigma}^d} \calO^{\sigma(\bfu, \bfv)} \cap C^{\gamma}.
$$
The subvariety 
$\theta_n(\calO^{\sigma(\bfu, \bfv)} \cap C^{\gamma}) \subset \calO^{u,v}$ is called $\mathcal{P}^{\bfw}_{\bfu, \bfv}$ in \cite[Section 5]{WY}, Thus one has 
$$
\calO^{u,v} = \bigsqcup_{\bfw} \mathcal{P}^{\bfw}_{\bfu, \bfv},
$$
where $\bfw$ runs over all double distinguished subexpressions. The ``piece" $ \mathcal{P}^{\bfw}_{\bfu, \bfv}$ is given in \cite[Theorem 5.2]{WY} the following parametrization.  Let $\gamma \in \Upsilon_{\bfu, \bfv, \sigma}^d$ be such that $(w_{(0)}^{-1}, ..., w_{(n)}^{-1}) = (\gamma^0, .... \gamma^n)$. For $k \in [1,n]$, let $\alpha_k$ be the simple root such that $s_{|i_k|} = s_{\alpha_k}$. For $z \in \IC$, let 
$$
g_k(z) = \left\{ \begin{array}{ll}
\overline{\gamma_{\bfv}^k} x_{\alpha_k}(z)\overline{\gamma_{\bfv}^k}^{-1}, & \epsilon_k = 1, k \notin J(\gamma)    \\
\overline{\gamma_{\bfu}^k} x_{-\alpha_k}(z)\overline{\gamma_{\bfu}^k}^{-1},  & \epsilon_k = -1, k \notin J(\gamma)    \\
e, & k \in J(\gamma),
\end{array} \right.
$$
recall (\ref{ga bfu and ga bfv}), and $g_{\bfw}(z_1, ..., z_n) = g_1(z_1) \cdots g_n(z_n)$, where 
\begin{equation} \label{z_k}
z_k \in \left\{ \begin{array}{ll}
\{0\}, & k \in J(\gamma),   \\
\IC^*, & k \notin I(\gamma),   \\
\IC, & k \in I(\gamma) \backslash J(\gamma). 
\end{array} \right.
\end{equation}
Then $(z_1, ..., z_n) \mapsto g_{\bfw}(z_1, ..., z_n) \cdot (\gamma_{\bfu}.B, \gamma_{\bfv}.B_-)$ is an isomorphism between $(\IC^*)^{|J_{\bfw}^0|} \times \IC^{| J_{\bfw}^-|}$ and $\mathcal{P}^{\bfw}_{\bfu, \bfv}$.

\begin{lem} \label{recovering WY}
Let $z = (z_1, ..., z_n)$ satisfying (\ref{z_k}). There exist signs $\varepsilon_1, ..., \varepsilon_n \in \{ \pm 1\}$ such that 
$$
\theta_n(u_{\gamma}(z)) = g_{\bfw}(\varepsilon_1z_1, ..., \varepsilon_nz_n) \cdot (\gamma_{\bfu}.B, \gamma_{\bfv}.B_-).
$$
\end{lem}
\begin{proof}
By Lemma \ref{key lemma}, one has 
\begin{align*}
\theta_n(u_{\gamma}(z)) & = (p_{\gamma, 1}(z_1) \cdots p_{\gamma, n}(z_n).B, q_{\gamma, 1}(z_1) \cdots q_{\gamma, n}(z_n).B_-)   \\
  & = p_{\gamma, 1}(z_1) \cdots p_{\gamma, n}(z_n) \bar{\gamma}_{\bfu}^{-1}(\gamma_{\bfu}.B, \gamma_{\bfv}.B_-)   \\
  & = q_{\gamma, 1}(z_1) \cdots q_{\gamma, n}(z_n) \bar{\gamma}_{\bfv}^{-1}(\gamma_{\bfu}.B, \gamma_{\bfv}.B_-).
\end{align*}
Let $H^{(2)} = \{ h \in H \mid h^2 = 1 \}$. It is enough to prove the existence of a $h_j \in H^{(2)}$ and $\varepsilon_1, ..., \varepsilon_j \in \{ \pm 1\}$ such that 
\begin{equation} \label{link with WY}
g_1(z_1) \cdots g_j(z_j) = p_{\gamma, 1}(\varepsilon_1 z_1) \cdots p_{\gamma, j}(\varepsilon_jz_j) \overline{\gamma_{\bfu}^j}^{-1}h_j, 
\end{equation}
for all $ j \in [1,n]$. 
If $j=1$, (\ref{link with WY}) is easily verified. So assume $j \geqslant 2$, and that (\ref{link with WY}) holds for $j-1$. Suppose first that $\epsilon_j = -1$. If $j \in J(\gamma)$, then $g_j(z_j) = e$ and $p_{\gamma, j}(z_j) = s_{|i_j|}$. So 
\begin{align*}
g_1(z_1) \cdots g_j(z_j)  & = g_1(z_1) \cdots g_{j-1}(z_{j-1})  =  p_{\gamma, 1}(\varepsilon_1 z_1) \cdots p_{\gamma, j-1}(\varepsilon_{j-1}z_{j-1}) \overline{\gamma_{\bfu}^{j-1}}^{-1}h_{j-1}   \\
  & =  p_{\gamma, 1}(\varepsilon_1 z_1) \cdots p_{\gamma, j-1}(\varepsilon_{j-1}z_{j-1})\bar{s}_{|i_j|}\overline{\gamma_{\bfu}^j}^{-1}h_j,
\end{align*}
for some $h_j \in H^{(2)}$, so we are done. If $j \notin J(\gamma)$, then 
\begin{align*}
g_1(z_1) \cdots g_j(z_j) & = p_{\gamma, 1}(\varepsilon_1 z_1) \cdots p_{\gamma, j-1}(\varepsilon_{j-1}z_{j-1}) \overline{\gamma_{\bfu}^{j-1}}^{-1}h_{j-1} \overline{\gamma_{\bfu}^j}x_{-\alpha_j}(z_j) \overline{\gamma_{\bfu}^j}^{-1}  \\
  & = p_{\gamma, 1}(\varepsilon_1 z_1) \cdots p_{\gamma, j-1}(\varepsilon_{j-1}z_{j-1}) \bar{\gamma}_jx_{-\alpha_j}(\varepsilon'z_j) \overline{\gamma_{\bfu}^j}^{-1} h'   \\
  & = p_{\gamma, 1}(\varepsilon_1 z_1) \cdots p_{\gamma, j-1}(\varepsilon_{j-1}z_{j-1}) x_{-\gamma_j\alpha_j}(\varepsilon_jz_j)\bar{\gamma}_j\overline{\gamma_{\bfu}^j}^{-1} h_j  \\
  & = p_{\gamma, 1}(\varepsilon_1 z_1) \cdots p_{\gamma, j}(\varepsilon_{j}z_{j}) \overline{\gamma_{\bfu}^j}^{-1} h_j ,
\end{align*}
for some $h', h_j \in H^{(2)}$, and $\varepsilon', \varepsilon_j \in \{ \pm 1 \}$. Thus we are done. Suppose now that $\epsilon_j = 1$. Observe that for any $k \in [1,n]$, 
$$
p_{\gamma, 1}(z_1) \cdots p_{\gamma, k}(z_k) \overline{\gamma_{\bfu}^{k}}^{-1} = q_{\gamma, 1}(z_1) \cdots q_{\gamma, k}(z_k) \overline{\gamma_{\bfv}^{k}}^{-1} h,
$$
for some $h \in H^{(2)}$. Hence one can proceed similarly by proving
$$
g_1(z_1) \cdots g_j(z_j) = q_{\gamma, 1}(\varepsilon_1 z_1) \cdots q_{\gamma, j}(\varepsilon_jz_j) \overline{\gamma_{\bfu}^j}^{-1}h_j, \;\;\; j \in [1,n],
$$
for some $h_j \in H^{(2)}$ and $\varepsilon_1, .... \varepsilon_j \in \{ \pm 1\}$. 
\end{proof}

We conclude that our parametrization of $C^{\gamma}$ recovers the one of $\mathcal{P}^{\bfw}_{\bfu, \bfv}$ when both $\bfu$ and $\bfv$ are reduced.

\section{Regular functions on $\calO^{\sigma(\bfu, \bfv)}$ defined by minors}

Let $(\bfu, \bfv, \sigma)$ be as in Definition \ref{sigma bfu bfv}. We introduce, for each distinguished $\sigma$-shuffled subexpression $\gamma$, a family of regular functions on $\calO^{\sigma(\bfu, \bfv)}$, and describe $\calO^{\sigma(\bfu, \bfv)} \cap C^{\gamma}$ using these functions. When $\gamma$ is positive, we relate this family of functions to the coordinates on $\calO^{\gamma}$.

\subsection{Generalised minors}

We recall the notion of generalised minors, and some of their basic properties that will be needed. See \cite{double} for further details. We assume from now on that $G$ is simply connected.    \\
\indent
The set $N_-HN$ is Zariski open in $G$. Moreover, recall that any element $x \in N_-HN$ is uniquely written as $x = [x]_-[x]_0[x]_+$, with $[x]_- \in N_-$, $ [x]_0 \in H$, $[x]_+ \in N$. Let $\alpha \in \Gamma$. For $x \in N_-HN$, define 
$$
\triangle_{\lambda_{\alpha}}(x) = [x]_0^{\lambda_{\alpha}}.
$$
The functions $\triangle_{\lambda_{\alpha}}$ extend to regular functions on $G$. For $G = SL(n,\IC)$, these are simply the principal $i \times i$ minor of a matrix $x \in G$, $i \in [1,n-1]$. \\
\indent
An equivalent definition of $\triangle_{\lambda_{\alpha}}$ is the following. Let $V_{\alpha}$ be the irreducible representation of $G$ of highest weight $\lambda_{\alpha}$, and let $v_{\alpha} \in V_{\alpha}$ be a highest weight vector. For any $a \in V$, let $\xi_{\alpha}(a)$ be the coefficient of $v_{\alpha}$ in the expansion of $a$ in any basis consisting of $v_{\alpha}$ and weight vectors. Then 
$$
\triangle_{\lambda_{\alpha}}(x) = \xi_{\alpha}(x \cdot v_{\alpha}), \;\;\; x \in G. 
$$
\indent
For any $u,v \in W$, the corresponding \textit{generalised minor} is the regular function on $G$ given by 
$$
\triangle_{u\lambda_{\alpha}, v\lambda_{\alpha}}(x) = \triangle_{\lambda_{\alpha}}(\bar{u}^{-1}x\bar{v}).
$$
It can be checked that $\triangle_{u\lambda_{\alpha}, v\lambda_{\alpha}}$ depends only on $u\lambda_{\alpha}$ and $v\lambda_{\alpha}$, thus making the notation consistent. \\
\indent
If $G = SL(n, \IC)$, the Weyl group of $G$ is naturally identified with the symmetric group $S_n$. Then $\triangle_{u\lambda_{\alpha}, v\lambda_{\alpha}}(x)$ is the minor of $x \in G$ formed by lines $u([1,i])$ and columns $v([1,i])$, where $\alpha$ is the $i$'th simple root. 

\begin{lem} \label{zlocus of triangle}
\cite[Proposition 2.4]{double}
Let $\alpha \in \Gamma$. The zero locus of $\triangle_{\lambda_{\alpha}}$ is precisely $\overline{B_-s_{\alpha}B}$.
\end{lem}

\begin{lem} \label{formula minor}
Let $\alpha \in \Gamma$. Then for any $z \in \IC$ and any $g \in G$, one has 
\begin{align*}
\triangle_{\lambda_{\alpha}}(x_{\alpha}(z)g) & = \triangle_{\lambda_{\alpha}}(g) + z \triangle_{s_{\alpha}\lambda_{\alpha}, \lambda_{\alpha}}(g)    \\
\triangle_{\lambda_{\alpha}}(gx_{-\alpha}(z)) & = \triangle_{\lambda_{\alpha}}(g) + z \triangle_{\lambda_{\alpha},s_{\alpha} \lambda_{\alpha}}(g).
\end{align*}
\end{lem} 
\begin{proof}
The two formulas can be proved similarly, so we will only prove the first one. 
Let $V$ be the irreducible representation of $G$ with highest weight $\lambda_{\alpha}$, and let $v \in V$ be a highest weight vector. For any $g \in G$, one has 
$$
g\cdot v = \sum_{w\lambda_{\alpha} \mid w \in W} \triangle_{w\lambda_{\alpha}, \lambda_{\alpha}}(g) \bar{w}\cdot v + R,
$$
where $R$ lies in the direct sum of non-extremal weight spaces. Thus for $z \in \IC$, 
$$
x_{\alpha}(z)g\cdot v =  \sum_{w\lambda_{\alpha} \mid w \in W} \triangle_{w\lambda_{\alpha}, \lambda_{\alpha}}(g) x_{\alpha}(z)\bar{w}\cdot v + x_{\alpha}(z)R.
$$
For any $a \in V$, let $\xi(a)$ be the coefficient of $v$ in the expansion of $a$ in any basis consisting of $v$ and weight vectors. Then 
\begin{align*}
\triangle_{\lambda_{\alpha}}(x_{\alpha}(z)g) & = \xi(x_{\alpha}(z)g\cdot v) = 
\sum_{w\lambda_{\alpha} \mid w \in W} \triangle_{w\lambda_{\alpha}, \lambda_{\alpha}}(g) \xi(x_{\alpha}(z)\bar{w}\cdot v) + \xi(x_{\alpha}(z)R)   \\
  & = \triangle_{\lambda_{\alpha}}(g) + z \triangle_{s_{\alpha}\lambda_{\alpha}, \lambda_{\alpha}}(g).
\end{align*}
\end{proof}

\subsection{Families of regular functions on $\calO^{\sigma(\bfu, \bfv)}$ defined by minors}

Let $(\bfu, \bfv, \sigma)$ be as in Definition \ref{sigma bfu bfv} and $\gamma$ be a $\sigma$-shuffled subexpression of $(\bfu, \bfv)$. Recall that $\gamma$ is \textit{distinguished} if  $C^{\gamma} \cap \calO^{\sigma(\bfu, \bfv)} \neq \emptyset$. This is equivalent, by Lemma \ref{J subset I}, to  $J(\gamma) \subset I(\gamma)$, recall Definitions \ref{J gamma} and \ref{I gamma}. If $\gamma$ is distinguished, we denote 
\begin{equation} \label{K gamma}
K(\gamma) = I(\gamma) \backslash J(\gamma). 
\end{equation}
Let $(z_1, ..., z_n)$ be the coordinates on $\calO^{\sigma(\bfu, \bfv)}$ defined in (\ref{u_gamma z}). To lighten the notation we write
$$
\calO = \calO^{\sigma(\bfu, \bfv)},
$$
and for $j \in [1,n]$, 
$$
p_j(z_j) = p_{\sigma(\bfu, \bfv), j}(z_j), \;\;\; q_j(z_j) = q_{\sigma(\bfu, \bfv), j}(z_j), \;\;\; u(z) = u_{\sigma(\bfu, \bfv)}(z).
$$
In particular, one has 
$$
p_j(z_j) \left\{ \begin{array}{ll}
 = x_{\alpha_j}(z_j) \bar{\delta}_j, & \text{ if } \epsilon_j = -1   \\
\in N,   & \text{ if } \epsilon_j = 1,
\end{array} \right.     \;\;
q_j(z_j) \left\{ \begin{array}{ll}
 \in N_-, & \text{ if } \epsilon_j = -1   \\
= x_{-\alpha_j}(z_j) \bar{\delta}^{-1}_j,   & \text{ if } \epsilon_j = 1.
\end{array} \right. 
$$

\begin{defn} \label{g_j(z_1, ..., z_j)}
For $j \in [1,n]$, let
$$
g_j(z_1, ..., z_j) =  \left\{  \begin{array}{ll}
(q_1(z_1) \cdots q_{j-1}(z_{j-1}))^{-1}p_1(z_1) \cdots p_j(z_j)  \in G, & \text{ if } \epsilon_j = -1   \\
(q_1(z_1) \cdots q_j(z_j))^{-1}p_1(z_1) \cdots p_{j-1}(z_{j-1})  \in G, & \text{ if } \epsilon_j = 1.
\end{array} \right. 
$$
For every $\gamma \in \Upsilon_{\bfu, \bfv, \sigma}^d$, and $j \in [1,n]$, define the regular function on $\calO$,
$$
\psi_{\gamma, j}(z_1, ..., z_j) = \left\{ \begin{array}{ll}
\triangle_{\lambda_j} \left( \overline{\gamma^{j-1}}^{-1}g_j(z_1, ..., z_j)\right), & \text{ if } \epsilon_j = -1 \\
\triangle_{\lambda_j} \left( g_j(z_1, ..., z_j)\overline{\gamma^{j-1}}^{-1}  \right), & \text{ if } \epsilon_j = 1,
\end{array} \right.  
$$
where $\lambda_j = \lambda_{\alpha_j}$ for $j \in [1,n]$. 
\end{defn}

\begin{exa} \label{example minors}
Let $G = SL(4, \IC)$ and $\sigma$ be the $(4,8)$-shuffle defined by 
$$
\begin{array}{llllllllll}
\epsilon(\sigma) & =  &  (-1, & -1, & 1, & 1, & -1, & 1, & -1, & 1).
\end{array}
$$
Let 
$$
\begin{array}{llllllllll}
\sigma(\bfu, \bfv) & = & (\gs_2, & \gs_3, & \gs_3, & \gs_1, & \gs_1, & \gs_2, & \gs_3, & \gs_1)
\end{array}
$$
and consider the two $\sigma$-shuffled subexpressions 
$$
\gamma = (\gs_2, e, \gs_3, e, e, e, \gs_3, e) \text{ and } \eta = (e,\gs_3, \gs_3, \gs_1, \gs_1, e, e, e).
$$
Then $\gamma$ is positive and $\eta$ is distinguished, but not positive. Indeed, one has $J(\eta) = \{2, 4 \} \subset I(\eta) = \{2,3,4,5\}$. One has 
$$
\begin{array}{ll}
\psi_{\gamma, 1} =  z_1    , & \psi_{\gamma, 2} =  z_2,     \\
\psi_{\gamma, 3} = z_2z_3-z_1     , & \psi_{\gamma, 4} = z_4,      \\
\psi_{\gamma, 5} =  z_4z_5-z_1    , & \psi_{\gamma, 6} = z_2z_5z_6- z_4z_5-z_2z_3 + z_1,      \\
\psi_{\gamma, 7} = (z_2z_6 - z_4)z_7 - z_6     , & \psi_{\gamma, 8} =   (z_4z_5-z_1)z_8 - (z_4z_7+z_6)z_5 + z_1z_7 + z_3, 
\end{array}
$$
and 
$$
\begin{array}{ll}
\psi_{\eta, 1} =  z_1    , & \psi_{\eta, 2} =  z_2,     \\
\psi_{\eta, 3} = z_3     , & \psi_{\eta, 4} = z_4,      \\
\psi_{\eta, 5} = z_5    , & \psi_{\eta, 6} = z_1z_6 - z_4z_3,      \\
\psi_{\eta, 7} = (z_2z_3-z_1)z_7- z_3     , & \psi_{\eta, 8} =  \psi_{\gamma, 8}.  
\end{array}
$$
\end{exa}

\begin{pro} \label{pro zerenonzero}
Let $u(z) \in \calO$ and $\gamma \in \Upsilon_{\bfu, \bfv, \sigma}^d$. Then $u(z) \in C^{\gamma}$ if and only if 
\begin{equation}  \label{zero nonzero}
\left\{ \begin{array}{ll}
\psi_{\gamma, j}(z_1, ..., z_j) = 0, & \text{ if } j \in J(\gamma)   \\
\psi_{\gamma, j}(z_1, ..., z_j) \neq 0, & \text{ if } j \notin I(\gamma).
\end{array} \right. 
\end{equation}
\end{pro}
\begin{proof}
By definition of $C^{\gamma}$ (recall Definition \ref{C^gamma1}), $u(z) \in C^{\gamma}$ if and only if 
$$
(q_1(z_1) \cdots q_j(z_j))^{-1}p_1(z_1) \cdots p_j(z_j) \in B_- \gamma^j B, \;\; j \in [1,n], 
$$
which is equivalent to 
\begin{equation} \label{condition of g_j}
g_j(z_1, ..., z_j)  \in B_- \gamma^j B,
\end{equation}
for any $j \in [1,n]$. Suppose that $u(z) \in C^{\gamma}$, and let $j \notin I(\gamma)$. Then $\gamma^j = \gamma^{j-1}$, and so 
$$
\left\{ \begin{array}{l}
\overline{\gamma^{j-1}}^{-1}g_j(z_1, ..., z_j)    \\
g_j(z_1, ..., z_j) \overline{\gamma^{j-1}}^{-1}
\end{array} \right.
\in N_-HN. 
$$
Thus one has $\psi_{\gamma, j}(z_1, ..., z_j) \neq 0$. Let now $j \in J(\gamma)$. Since $\gamma$ is distinguished, this implies that $\gamma_j = \delta_j$, thus $(\gamma^{j-1})^{-\epsilon_j} \alpha_j > 0$. Assume that $\epsilon_j = -1$, the case $\epsilon_j = 1$ being similar. One has
$$
g_j(z_1, ..., z_j) \in N_- \gamma^jB = (N_- \cap \gamma^j N_- (\gamma^j)^{-1})\gamma^{j-1} \delta_jB.
$$
It follows that 
\begin{align*}
\overline{\gamma^{j-1}}^{-1}g_j(z_1, ..., z_j)  \in \left( (\gamma^{j-1})^{-1}N_- \gamma^{j-1} \cap \delta_j N_- \delta_j \right) \delta_j B  \subset N_- \delta_j B.
\end{align*}
Thus by Lemma \ref{zlocus of triangle}, one must have $\psi_{\gamma, j}(z_1, ..., z_j) = 0$. \\
\indent
Conversely, suppose that $u(z) \in \calO$ satisfies (\ref{zero nonzero}). We need to show that (\ref{condition of g_j}) holds for every $j \in [1,n]$. Let $j = 1$. Since $\gamma^0 = e$, one has 
$$
\psi_{\gamma, 1}(z_1) = \triangle_{\lambda_1}(x_{\alpha_1}(z_1) \bar{\delta}_1) = z_1.
$$
If $\gamma_1 = e$, one has $z_1 \neq 0$, hence $g_1(z_1) \in B_- B = B_- \gamma^1B$. If $\gamma_1 = \delta_1$, then $1 \in J(\gamma)$, so $z_1 = 0$. Thus $g_1(z_1) \in B_- \delta_1B = B_- \gamma^1B$. So assume that $j \geqslant 2$, and that (\ref{condition of g_j}) holds for $j-1$. Once again both cases $\epsilon_j = -1$ and $\epsilon_j = 1$ are similar, so we will assume that $\epsilon_j = -1$. \\
\noindent
\textbf{Case $j \notin I(\gamma)$:} \\
\noindent
Since $\gamma$ is distinguished, $j \notin J(\gamma)$, so $\gamma^j \alpha_j > 0$, which means $\gamma^j \delta_j > \gamma^j$. Then 
\begin{align*}
g_j(z_1, ..., z_j) & =(q_1(z_1) \cdots q_{j-1}(z_{j-1}))^{-1}p_1(z_1) \cdots p_{j-1}(z_{j-1})  p_j(z_j) \\
  & \in B_- \gamma^{j-1}B\delta_jB = B_-\gamma^{j-1}B \cup B_- \gamma^{j-1} \delta_jB. 
\end{align*}
If $g_j(z_1, ..., z_j) \in B_- \gamma^{j-1} \delta_jB$, then $\overline{\gamma^{j-1}}^{-1}g_j(z_1, ..., z_j) \in B_-\delta_jB$. This implies, by Lemma \ref{zlocus of triangle},  that $\psi_{\gamma, j}(z_1, ..., z_j) = 0$, which is a contradiction. Thus 
$$
g_j(z_1, ..., z_j) \in B_- \gamma^{j-1}B = B_- \gamma^jB.  
$$
\textbf{Case $j \in J(\gamma)$:}   \\
\noindent
Since $\gamma$ is distinguished, $\gamma_j = \delta_j$. Thus $\gamma^{j-1}\delta_j > \gamma^{j-1}$. One gets 
\begin{align*}
g_j(z_1, ...,z_j) & = (q_1(z_1) \cdots q_{j-1}(z_{j-1}))^{-1}p_1(z_1) \cdots p_{j-1}(z_{j-1})  p_j(z_j)  \\
  & \in B_- \gamma^{j-1}B\delta_jB = B_- \gamma^{j-1}B \cup B_- \gamma^{j-1} \delta_jB. 
\end{align*}
If $g_j(z_1, ..., z_j) \in B_- \gamma^{j-1}B$, then $\overline{\gamma^{j-1}}^{-1}g_j(z_1, ..., z_j) \in B_-B$. This implies that $\psi_{\gamma, j}(z_1, ..., z_j) \neq 0$, which is a contradiction. Thus 
$$
g_j(z_1, ..., z_j) \in B_- \gamma^{j-1}\delta_jB = B_- \gamma^jB.  
$$
\textbf{Case $j \in K(\gamma)$:}   \\
\noindent
Since $\gamma$ is distinguished, $\gamma^j \alpha_j > 0$, so $\gamma^{j-1}\delta_j < \gamma^{j-1}$. Hence 
\begin{align*}
g_j(z_1, ..., z_j) & =(q_1(z_1) \cdots q_{j-1}(z_{j-1}))^{-1}p_1(z_1) \cdots p_{j-1}(z_{j-1})  p_j(z_j)  \\
  & \in B_- \gamma^{j-1}B\delta_jB =  B_- \gamma^{j-1}\delta_j B = B_- \gamma^jB. 
\end{align*}
Thus (\ref {condition of g_j}) holds for all $j \in [1,n]$. 
\end{proof}

\begin{lem} \label{L_j M_j}
Let $\gamma \in \Upsilon_{\bfu, \bfv, \sigma}^d$ and $j \in [1,n]$. There exist polynomials $L_j, M_j \in \IC[z_1, ..., z_{j-1}]$ such that 
$$
\psi_{\gamma, j}(z_1, ..., z_j) = L_j + z_jM_j. 
$$
Moreover, $M_j$ vanishes nowhere on $C^{\gamma}$. 
\end{lem}
\begin{proof}
If $j = 1$, then $\psi_{\gamma, 1}(z_1) = z_1$. So assume now that $j \geqslant 2$, and that $\epsilon_j = -1$, the other case being similar. One has $g_j(z_1, ..., z_j) = x_{j-1} x_{\alpha_j}(z_j)\bar{\delta}_j$, where 
$$
x_{j-1} = (q_1(z_1) \cdots q_{j-1}(z_{j-1}))^{-1}p_1(z_1) \cdots p_{j-1}(z_{j-1})
$$
depends only on $z_1, ..., z_{j-1}$. Then, using Lemma \ref{formula minor}, 
\begin{align*}
\psi_{\gamma, j}(z_1, ..., z_j) & = \triangle_{\lambda_j}(\overline{\gamma^{j-1}}^{-1}x_{j-1} \bar{\delta}_j x_{-\alpha_j}(-z_j))    \\
  & = \triangle_{\lambda_j}(\overline{\gamma^{j-1}}^{-1}x_{j-1} \bar{\delta}_j ) + z_j\triangle_{\lambda_j}(\overline{\gamma^{j-1}}^{-1}x_{j-1}).
\end{align*}
Thus one lets 
$$
L_j = \triangle_{\lambda_j}(\overline{\gamma^{j-1}}^{-1}x_{j-1} \bar{\delta}_j ) \;\; \text{and } \; M_j = \triangle_{\lambda_j}(\overline{\gamma^{j-1}}^{-1}x_{j-1}).
$$
Suppose now that $u(z) \in C^{\gamma}$. Then in particular, $x_{j-1} \in B_- \gamma^{j-1}B$, so that $\overline{\gamma^{j-1}}^{-1}x_{j-1} \in B_-B$. Hence $M_j(z_1, ..., z_{j-1}) \neq 0$.  
\end{proof}

\begin{cor} \label{O cap C^gamma2}
The map 
$$
\calO \cap C^{\gamma} \rightarrow \IC^{|K(\gamma)|} \times (\IC^*)^{|I(\gamma)^c|}, \; u(z) \mapsto ((\psi_{\gamma, j})_{j \in K(\gamma)}, (\psi_{\gamma, j})_{j \notin I(\gamma)}),
$$
where $I(\gamma)^c = [1,n] \backslash I(\gamma)$, is an isomorphism. 
\end{cor}
\begin{proof}
Indeed by Lemma \ref{L_j M_j}, for any value
$$
((x_j)_{j \in K(\gamma)}, (x_j)_{j \notin I(\gamma)}) \in \IC^{|K(\gamma)|} \times (\IC^*)^{|I(\gamma)^c|},
$$ 
one can uniquely solve $\psi_{\gamma, j}(z_1, ..., z_j) = x_j$. 
\end{proof}

\begin{exa}
Let $\gamma$ and $\eta$ be as in Example \ref{example minors}. By Proposition \ref{pro zerenonzero}, $\calO \cap C^{\gamma}$ is given by 
$$
z_1 = z_3 = 0, \;\;\; z_2, z_4, z_5, z_2z_6 - z_4, z_4(z_8 - z_7) + z_6 \neq 0.
$$
One can check directly that $\calO \cap C^{\gamma}$ is isomorphic to $(\IC^*)^5$ via 
$$
(z_1, ..., z_8) \in \calO \cap C^{\gamma} \mapsto (z_2, z_4, z_5, z_2z_6 - z_4, z_4(z_8 - z_7) + z_6 ) \in (\IC^*)^5.  
$$
Similarly, $\calO \cap C^{\eta}$ is given by 
$$
z_2 = z_4 = 0, \;\;\; z_1, z_6, z_1z_7+z_3, z_1(z_7-z_8)- z_6z_5+z_3 \neq 0, 
$$
and $\calO \cap C^{\eta}$ is isomorphic to $(\IC^*)^4 \times \IC^2$ via 
$$
(z_1, ..., z_8) \in \calO \cap C^{\eta} \mapsto \left( (z_1, z_6, z_1z_7+z_3, z_1(z_7-z_8)- z_6z_5+z_3), (z_3, z_5) \right) \in (\IC^*)^4 \times \IC^2.
$$
\end{exa}

\subsection{Factorization problem for a positive subexpression}

Fix now a positive $\sigma$-shuffled subexpression $\gamma$. Let $\xi = (\xi_1, ..., \xi_n)$ be the coordinates on $\calO^{\gamma}$ defined in (\ref{u_gamma z}). By Proposition \ref{O cap C^gamma}, $u_{\gamma}(\xi) \in \calO \cap C^{\gamma}$ if and only if $\xi_j = 0$ when $j \in J(\gamma)$, and $\xi_j \neq 0$ when $j \notin J(\gamma)$. Hence 
$$
\calO \cap C^{\gamma} \rightarrow (\IC^*)^{|J(\gamma)^c|}, \;\;\; u_{\gamma}(\xi) \mapsto (\xi_j)_{j \notin J(\gamma)},
$$
is an isomorphism. On the other hand, by Corollary \ref{O cap C^gamma2}, 
$$
\calO \cap C^{\gamma} \rightarrow (\IC^*)^{|J(\gamma)^c|}, \;\;\; u(z) \mapsto (\psi_{\gamma, j})_{j \notin J(\gamma)}
$$
is another set of coordinates on $\calO \cap C^{\gamma}$. We show now that these two systems of coordinates are related by a triangular matrix of monomials. 

\begin{defn}
For any $1 \leqslant k <  j \leqslant n$, define 
$$
\bfu_{(k,j]} = \prod_{k < i \leqslant j, \epsilon_i = -1} \delta_i, \;\;\;  \bfv_{(k,j]} = \prod_{k < i \leqslant j, \epsilon_i = 1} \delta_i,
$$
where the index in both products is increasing, and it is understood that $\bfu_{(k,j]}$ or $ \bfv_{(k,j]} = e$, if the set over which the product is defined is empty. 
\end{defn}

\begin{thm} \label{m_j,k}
Let $\gamma \in \Upsilon_{\bfu, \bfv, \sigma}^+$, and $j \notin J(\gamma)$. Then over $\calO \cap C^{\gamma}$, one has 
$$
\psi_{\gamma, j} = \prod_{1 \leqslant k \leqslant j, k \notin J(\gamma)} \xi_k^{-m_{j,k}},
$$
where 
$$
m_{j,k} = \left\{ \begin{array}{ll}
( \bfv_{(k,j]} \gamma^j \lambda_j, \calpha_k ), & \text{ if } \epsilon_j = -1, \epsilon_k = 1  \\
( \bfu_{(k,j]} \lambda_j, \calpha_k ), & \text{ if } \epsilon_j = -1, \epsilon_k = -1  \\
( \bfv_{(k,j]} \lambda_j, \calpha_k ), & \text{ if } \epsilon_j = 1, \epsilon_k = 1  \\
( \bfu_{(k,j]} (\gamma^j)^{-1} \lambda_j, \calpha_k ), & \text{ if } \epsilon_j = 1, \epsilon_k = -1.  \\
\end{array} \right. 
$$
\end{thm}
\begin{proof}
Let $z = (z_1, ..., z_n), \xi = (\xi_1, ..., \xi_n) \in \IC^n$ and suppose that $u(z) = u_{\gamma}(\xi)$. Then  there exist elements $b_j \in B$, $b_{-j} \in B_-$, for $j \in [1,n]$, such that 
$$
\left\{ \begin{array}{ll}
p_{\gamma, 1}(\xi_1) &  = p_1(z_1)b_1  \\
b_1p_{\gamma, 2}(\xi_2) & = p_2(z_2) b_2   \\
  & \vdots    \\
b_{n-1}p_{\gamma, n}(\xi_n) & = p_n(z_n) b_n,
\end{array} \right.
\;\;
\left\{ \begin{array}{ll}
q_{\gamma, 1}(\xi_1) &  = q_1(z_1)b_{-1}  \\
b_{-1}q_{\gamma, 2}(\xi_2) & = q_2(z_2) b_{-2}   \\
  & \vdots    \\
b_{-(n-1)}q_{\gamma, n}(\xi_n) & = q_n(z_n) b_{-n}.
\end{array} \right.
$$
Write 
$$
b_j = t_j a_j, \;\;\; b_{-j} = t_{-j} a_{-j}, 
$$
with $t_j, t_{-j} \in H$, $a_j \in N$, and $a_{-j} \in N_-$, $j \in [1,n]$. 
Let now $j \notin J(\gamma)$ and suppose first that $\epsilon_j = -1$. If  $u(z) = u_{\gamma}(\xi) \in C^{\gamma}$, then by Lemma \ref{key lemma}, and Remark \ref{gamma tilde = gamma bar},
\begin{align*}
g_j(z_1, ..., z_j) & = q_j(z_j) (q_1(z_1) \cdots q_j(z_j))^{-1}p_1(z_1) \cdots p_j(z_j)    \\
  & = q_j(z_j) \left( q_{\gamma, 1}(\xi_1) \cdots q_{\gamma, j}(\xi_j) b_{-j}^{-1}\right)^{-1} p_{\gamma, 1}(\xi_1) \cdots p_{\gamma, n}(\xi_n) b_j^{-1}    \\
  & = q_j(z_j) b_{-j} \overline{\gamma^j} b_j^{-1}.
\end{align*}
Similarly, if $\epsilon_j = 1$, then 
\begin{align*}
g_j(z_1, ..., z_j) & =  b_{-j} \overline{\gamma^j} b_j^{-1} p_j(z_j)^{-1}. 
\end{align*}
Recall that since $\gamma$ is positive, one has $\gamma^j = \gamma^{j-1}$. So if $\epsilon_j = -1$, 
\begin{align}
\psi_{\gamma, j} & = \triangle_{\lambda_j}(\overline{\gamma^{j-1}}^{-1} g_j(z_1, ..., z_j)) = \triangle_{\lambda_j}(\overline{\gamma^{j-1}}^{-1} q_j(z_j) b_{-j} \overline{\gamma^j} b_j^{-1}) \notag   \\
  & = \triangle_{\lambda_j}(\overline{\gamma^j}^{-1} q_j(z_j) b_{-j} \overline{\gamma^j} b_j^{-1})  \notag  \\
  & = t_{-j}^{\gamma^j\lambda_j} t_j^{-\lambda_j}. \label{psi epsilon = -1}
\end{align}
Similarly, if $\epsilon_j = 1$, one gets 
\begin{equation} \label{psi epsilon = 1}
\psi_{\gamma, j} = t_{-j}^{\lambda_j} t_j^{-(\gamma^j)^{-1} \lambda_j}.
\end{equation}

Suppose now that $\epsilon_1 = -1$. Since $q_{\gamma, 1}(\xi_1)$ and $q_1(z_1)$ both lie in $N_-$, one has $t_{-1} = e$. If $\gamma_1 = \delta_1$, one has $1 \in J(\gamma)$, and $p_{\gamma, 1}(\xi_1) = \bar{\delta}_1$. Thus $t_1 = e$. If $\gamma_1 =e$, then 
$$
p_{\gamma, 1}(\xi_1) = x_{-\alpha_1}(\xi_1) = x_{\alpha_1}(\xi_1^{-1}) \bar{\delta}_1 \calpha_1(\xi_1)x_{\alpha_1}(\xi_1^{-1}),
$$
and so $t_1 = \calpha_1(\xi_1)$. Similarly, if $\epsilon_1 = 1$, then $t_1 = e$, and 
$$
t_{-1} = \left\{ \begin{array}{ll}
e, & \text{ if } \gamma_1 = \delta_1   \\
\calpha_1(\xi_1^{-1}), & \text{ if } \gamma_1  = e. 
\end{array} \right. 
$$
We claim now that for all $j \in [2,n]$ and $u(z) = u_{\gamma}(\xi) \in \calO \cap C^{\gamma}$,
\begin{equation} \label{t_j}
t_j = \left\{ \begin{array}{ll}
t_{j-1}^{\delta_j}, & \text{ if } \epsilon_j = -1, \; j \in J(\gamma)    \\
\calpha_j(\xi_j)t_{j-1}^{\delta_j}, & \text{ if } \epsilon_j = -1, \; j \notin J(\gamma)    \\
t_{j-1}, & \text{ if } \epsilon_j = 1. 
\end{array} \right. 
\end{equation}
and 
\begin{equation} \label{t_{-j}}
t_{-j} = \left\{ \begin{array}{ll}
t_{-(j-1)}^{\delta_j}, & \text{ if } \epsilon_j = 1, \; j \in J(\gamma)    \\
\calpha_j(\xi_j^{-1})t_{-(j-1)}^{\delta_j}, & \text{ if } \epsilon_j = 1, \; j \notin J(\gamma)    \\
t_{-(j-1)}, & \text{ if } \epsilon_j = -1. 
\end{array} \right. 
\end{equation}
Both (\ref{t_j}) and (\ref{t_{-j}}) arise from similar calculations, so we will concentrate on (\ref{t_j}). Suppose that $\epsilon_j = -1$ and $j \in J(\gamma)$. Then $p_{\gamma, j}(\xi_j) = \bar{\delta}_j$, so 
$$
b_{j-1}p_{\gamma, j}(\xi_j) = t_{j-1}a_{j-1} \bar{\delta}_j = x_{\alpha_j}(z_j) \bar{\delta}_j t_{j-1}^{\delta_j} a_j.
$$
Hence $t_j = t_{j-1}^{\delta_j}$. Suppose now that $j \notin J(\gamma)$. Write $a_{j-1} = x_{\alpha_j}(\xi_j')a_{j-1}'$, with $\xi_j' \in \IC$ and $a_{j-1}' \in N \cap \delta_jN \delta_j$. Then 
\begin{align*}
b_{j-1}p_{\gamma, j}(\xi_j) & = t_{j-1}a_{j-1} x_{-\alpha_j}(\xi_j)     \\
  & = t_{j-1} x_{\alpha_j}(\xi_j') x_{-\alpha_j}(\xi_j) \left( x_{-\alpha_j}(-\xi_j)a_{j-1}' x_{-\alpha_j}(\xi_j) \right)   \\
  & = t_{j-1} x_{\alpha_j}((1+\xi_j\xi_j')/\xi_j) \bar{\delta}_j \calpha_j(\xi_j) x_{\alpha}(\xi_j^{-1}) \left( x_{-\alpha_j}(-\xi_j)a_{j-1}' x_{-\alpha_j}(\xi_j) \right)    \\
  & = x_{\alpha_j}(z_j) \bar{\delta}_j t_{j-1}^{\delta_j} \calpha_j(\xi_j) a_j,
\end{align*}
so $t_j = \calpha_j(\xi_j)t_{j-1}^{\delta_j}$. Suppose now that $\epsilon_j = 1$. Then since $p_{\gamma, j}(\xi_j)$ and $p_j(z_j)$ lie in $N$, one has $t_j = t_{j-1}$. \\
\indent
Now (\ref{t_j}) and (\ref{t_{-j}}) give 
\begin{align*}
t_j  & = \prod_{k=1, \epsilon_k = -1, k \notin J(\gamma)}^j \calpha_k(\xi_k)^{\bfu_{(k,j]}}   \\
t_{-j}  & = \prod_{k=1, \epsilon_k = 1, k \notin J(\gamma)}^j \calpha_k(\xi_k^{-1})^{\bfv_{(k,j]}}. 
\end{align*}
Substituting in (\ref{psi epsilon = -1}) and (\ref{psi epsilon = 1}) yields the result. 
\end{proof}

\begin{defn} \label{M_gamma}
For any $\gamma \in \Upsilon_{\bfu, \bfv, \sigma}^+$, let $M_{\gamma} = (m_{j,k})_{j,k \notin J(\gamma)}$ be the $|J(\gamma)^c| \times |J(\gamma)^c|$ lower triangular matrix where $m_{j,k}$ is as in Theorem \ref{m_j,k}, for $k \leqslant j$, $j,k \notin J(\gamma)$. 
\end{defn}

Notice that $M_{\gamma}$ has diagonal entries $1$. Informally, we write Theorem \ref{m_j,k} as 
$$
\psi_{\gamma} = \xi^{M_{\gamma}}.
$$

\subsection{On the inverse of $M_{\gamma}$}

Let $L_{\gamma}$ be the inverse of $M_{\gamma}$. The relation 
$$
\xi = \psi_{\gamma}^{L_{\gamma}}
$$
can be thought as an analogous of the inverse factorization problem of  Fomin and Zelevinsky, see \cite{double}. That is, the local coordinates $(\xi_k)_{k \notin J(\gamma)}$ on $\calO \cap C^{\gamma}$ are being written in terms of regular functions $(\psi_{\gamma, j})_{j \notin J(\gamma)} \in \IC[\calO]$. \\
\indent
The following Lemma \ref{inv matrix} provides an inductive formula to express the entries of $L_{\gamma}$.

\begin{lem} \label{inv matrix}
Let $V$ be a vector space over a field $\IK$, $v_1, ..., v_n \in V$, and $\varphi_1, ..., \varphi_n \in V^*$. For $1 \leq k < j \leq n$, Suppose given operators $M_{jk} \in \End(V)$ and let $(m_{jk})_{j,k = 1, ..., n} \in \g \gl(n, \IK)$ be the lower triangular matrix with diagonal entries $1$ and 
$$
m_{jk} = (M_{jk} v_j, \varphi_k), \;\;\; 1 \leqslant k < j \leqslant n.
$$
Let $(l_{jk})_{j,k = 1, ..., n}$ be the inverse of $(m_{jk})_{j,k = 1, ..., n}$. Then there exist operators $L_{ji} \in \End(V)$ such that 
$$
l_{jk} = (L_{jk} v_j, \varphi_k), \;\;\;  1 \leqslant k < j \leqslant n. 
$$
\end{lem}
\begin{proof}
For $j-k = 1$, one lets $L_{jk} = -M_{jk}$. Thus assume $j-k > 1$ and the operators $L_{j'k'}$ defined for $j'-k' < j-k$. Then 
\begin{align*}
l_{jk} & = -\left( \sum_{i = k+1}^{j-1} l_{ji}m_{ik} \right) - m_{jk}   \\
   & =  -\left( \sum_{i = k+1}^{j-1} ((L_{ji}v_j, \varphi_i)M_{ik}v_i, \varphi_k) \right) - m_{jk}   \\
   & = \left(- \left( \sum_{i = k+1}^{j-1} (L_{ji}v_j, \varphi_i)M_{ik}v_i \right) - M_{jk}v_j, \varphi_k   \right).
\end{align*}
Thus set 
$$
L_{jk}x = - \left( \sum_{i = k+1}^{j-1} (L_{ji}x, \varphi_i)M_{ik}v_i \right) - M_{jk}x, \;\;\; x \in V.
$$
\end{proof}

By Lemma \ref{inv matrix}, one can inductively define operators $L_{jk} \in \End(\gh^*)$, for $k< j$, $j,k \notin J(\gamma)$, such that 
\begin{equation} \label{L_ji}
(L_{\gamma})_{j,k} = (L_{jk} \lambda_j, \calpha_k).
\end{equation}
In the case when $\bfv = \emptyset$, that is one only has a single Bott Samelson variety $Z_{\bfu}$, the operators in (\ref{L_ji}) can be expressed in simple way.

\begin{defn}
For any $\alpha \in \Gamma$, define $r_{\alpha} \in \End(\gh^*)$ by 
$$
r_{\alpha}x = x - (x, \calpha_i) \lambda_{\alpha}, \;\;\; x \in \gh^*.
$$
\end{defn}


\begin{pro} \cite{L}
Suppose that $\bfv = \emptyset$. For $i \in [1,n]$, let $r_i = r_{\alpha_i}$ and 
$$
\tilde{\gamma}_i =  \left\{ \begin{array}{ll}
s_i, & \text{ if } i \in J(\gamma)     \\
r_i, & \text{ if } i \notin J(\gamma). 
\end{array} \right. 
$$
For $k < j$ and $k,j \notin J(\gamma)$, one has 
$$
(L_{\gamma})_{k,j} = - (\tilde{\gamma}_{k+1} \cdots \tilde{\gamma}_{j-1}s_j \lambda_j, \calpha_k).
$$
\end{pro}

\section*{Acknowledgments}

The author would like to thank gratefully Jiang-Hua Lu for her help and encouragements. This work was completed while the author was supported by a University of Hong Kong Post-graduate Studentship and  by the HKRGC grant HKU 704310P.

\newpage
\bibliographystyle{plain} 

\addcontentsline{toc}{chapter}{Bibliography}

\end{document}